\documentclass{amsart}
\usepackage{amssymb,amscd}  

\newtheorem{theorem}{Theorem}[section]
\newtheorem{lemma}[theorem]{Lemma}
\newtheorem{corollary}[theorem]{Corollary}

\newtheorem{proposition}[theorem]{Proposition}

\theoremstyle{definition}
\newtheorem{definition}[theorem]{Definition}

\newtheorem{remark}[theorem]{Remark} 
\numberwithin{equation}{section}

\newcommand\C{\mathbf{C}} 

\newcommand\Q{\mathbf{Q}}
\newcommand\Z{\mathbf{Z}}
\newcommand\twomatr[4]{\begin{pmatrix}#1&#2\\ #3&#4 \end{pmatrix}}
\newcommand\stwomatr[4]{\left(\begin{smallmatrix}#1&#2\\  
                               #3&#4 \end{smallmatrix}\right)}
\newcommand\notdividing{{\not |\,}}
\newcommand\tensor{\otimes}
\newcommand\isomorphic{\cong}

\DeclareMathOperator{\Sym}{Sym}

\DeclareMathOperator{\tr}{tr}

\DeclareMathOperator{\linalgspan}{span}

\DeclareMathOperator{\Res}{Res}
\DeclareMathOperator{\Real}{Re}
\DeclareMathOperator{\Divisor}{div}

\newcommand\intersect{\cap}

\newcommand\abs[1]{{\left|#1\right|}}

\DeclareMathOperator{\Gal}{Gal}
\newcommand\compose{\circ}

\newcommand\Projective{{\bf P}} 
\newcommand\fieldk{k} 
\newcommand\fieldbigK{K}
\newcommand\fieldbigKbar{{\overline{\fieldbigK}}} 
\newcommand\fieldbigKl{{\fieldbigK_\myl}} 
\newcommand\E{E} 
\newcommand\Pzero{{\pointP_0}} 
\newcommand\myl{\ell} 
\newcommand\wtj{j} 
\newcommand\El{\E[\myl]} 
\newcommand\En{\E[n]}
\newcommand\Enl{\E[n\myl]}
\newcommand\divD{D} 
\newcommand\divE{E} 
\newcommand\thelambda{\lambda}
\newcommand\lambdaD{\thelambda_\divD}
\newcommand\lambdaE{\thelambda_\divE}

\newcommand\lambdaP{\thelambda_\pointP}
\newcommand\lambdaQ{\thelambda_\pointQ}
\newcommand\lambdaR{\thelambda_\pointR}
\newcommand\lambdaS{\thelambda_\pointS}
\newcommand\themu{\mu}
\newcommand\muD{\themu_\divD}

\newcommand\muP{\themu_\pointP}
\newcommand\muQ{\themu_\pointQ}
\newcommand\muR{\themu_\pointR}
\newcommand\thenu{\nu}
\newcommand\nuD{\thenu_\divD}

\newcommand\nuP{\thenu_\pointP}
\newcommand\nuQ{\thenu_\pointQ}
\newcommand\nuR{\thenu_\pointR}
\newcommand\thef{f}
\newcommand\fD{\thef_\divD}
\newcommand\fE{\thef_\divE}

\newcommand\fP{\thef_\pointP}
\newcommand\fQ{\thef_\pointQ}

\newcommand\theg{g}

\newcommand\gP{\theg_\pointP}
\newcommand\gQ{\theg_\pointQ}

\newcommand\pointP{P} 
\newcommand\pointQ{Q} 
\newcommand\pointR{R} 
\newcommand\pointS{S} 
\newcommand\pointT{T} 
\newcommand\xP{x_\pointP} 
\newcommand\yP{y_\pointP} 
\newcommand\zP{z_\pointP} 
\newcommand\xQ{x_\pointQ} 
\newcommand\yQ{y_\pointQ} 
\newcommand\xR{x_\pointR} 
\newcommand\yR{y_\pointR} 

\newcommand\ringR{{\mathcal{R}}}
\newcommand\ringRl{{\ringR_\myl}}
\newcommand\acoeff{a} 
\newcommand\bcoeff{b} 
\newcommand\scalarlambda{u} 
\newcommand\modforms{\mathcal{M}}
\newcommand\cuspforms{\mathcal{S}}

\newcommand\Half{\mathcal{H}} 
\newcommand\Gammal{{\Gamma(\myl)}} 
\newcommand\intm{m} 
\newcommand\intn{n}
\newcommand\torsi{{a_1}}
\newcommand\torsj{{a_2}}
\newcommand\Eis{G} 
\newcommand\Ej{\Eis_\wtj}
\newcommand\Eone{\Eis_1}
\newcommand\Etwo{\Eis_2}
\newcommand\Ethree{\Eis_3}
\newcommand\Efour{\Eis_4}
\newcommand\Esix{\Eis_6}
\newcommand\trstar{\tau_\pointR^*}
\newcommand\LL{\mathcal{L}}
\newcommand\LLhat{\hat{\mathcal{L}}} 
\newcommand\p{p} 
\newcommand\mubar{\overline{\mu}} 
\newcommand\boldmu{\boldsymbol{\mu}} 
\newcommand\Ohat{\hat{\mathcal{O}}} 
\newcommand\Espace{\mathcal{E}} 
\newcommand\zirc{0} 

\begin{document}

\title{Moduli interpretation of Eisenstein series}

\author{Kamal Khuri-Makdisi}
\address{Mathematics Department and Center for Advanced Mathematical Sciences,
American University of Beirut, Bliss Street, Beirut, Lebanon}
\email{kmakdisi@aub.edu.lb}
\thanks{September 22, 2011}

\begin{abstract}
Let $\ell \geq 3$.  Using the moduli interpretation, we define certain
elliptic modular forms of level $\Gamma(\ell)$ over
any field $k$ where $6\ell$ is invertible and $k$ contains
the $\ell$th roots of unity.  These forms generate a graded algebra 
$\mathcal{R}_\ell$, which, over $\mathbf{C}$, is generated by the
Eisenstein series of weight~$1$ on $\Gamma(\ell)$.  The main result of
this article is that, when $k=\mathbf{C}$, the ring $\mathcal{R}_\ell$
contains all modular forms on $\Gamma(\ell)$ in weights $\geq 2$.  The
proof combines algebraic and analytic techniques, including
the action of Hecke operators and nonvanishing of $L$-functions.  Our
results give a systematic method to produce models for the modular
curve $X(\ell)$ defined over the $\ell$th cyclotomic field, using only
exact arithmetic in the $\ell$-torsion field of a single
$\mathbf{Q}$-rational elliptic curve $E^0$.
\end{abstract}

\maketitle

\section{Introduction}
\label{section1}

Given a lattice $L \subset \C$, let $\wp$ and $\zeta$ be the
Weierstrass functions with respect to $L$.  A classical formula (see,
e.g., equation IV.3.6 of~\cite{Chandrasekharan}), which we reprove in
Corollary~\ref{corollary3.6} below, states that if
$\alpha,\beta,\gamma \in \C - L$ and $\alpha + \beta + \gamma = 0$,
then 
\begin{equation}
\label{equation1.1}
\frac{-1}{2} \cdot 
\frac{\wp'(\alpha) - \wp'(\beta)}{\wp(\alpha) - \wp(\beta)} 
= \zeta(\alpha) + \zeta(\beta) + \zeta(\gamma).
\end{equation}
Let us
temporarily call the above expression $\thelambda =
\thelambda_{\alpha, \beta, \gamma, L}$.  From the series for
$\zeta$, one can show that $\thelambda$ is equal to the
absolutely convergent series 
\begin{equation}
\label{equation1.2}
\thelambda
 =\zeta(\alpha) + \zeta(\beta) + \zeta(\gamma)
 = \sum_{\omega \in L}{}' \left(
           \frac{1}{\omega + \alpha}
         + \frac{1}{\omega + \beta}
         + \frac{1}{\omega + \gamma}
         - \frac{3}{\omega}
                        \right),
\end{equation}
where $\sum'$ means that one omits the term $3/\omega$
from the summand when $\omega = 0$.  The individual sums
such as $\sum_\omega 1/(\omega + \alpha)$ do not converge; however, if
$\alpha, \beta, \gamma \in \frac{1}{\myl}L$ for some $\myl$,
then the sums can be regularized by Hecke's method, and
$\thelambda$ is a suitable weight~$1$ Eisenstein series on $\Gammal$;
we prove this in Section~\ref{section2}.    Now view the
elliptic curve $\E = \C/L$ as a plane cubic (the Weierstrass model)
via $\wp$ and $\wp'$.  Then, essentially, the coordinates
$\wp(\alpha), \wp'(\alpha)$ of torsion points in $\El$
are Eisenstein series in weights $2$ and~$3$, while the weight~$1$
Eisenstein series $\thelambda$ is the slope of the line joining
the torsion points attached to $\alpha$ and $\beta$.
Hence Eisenstein series 
of weights $\leq 3$ can be computed from the Weierstrass model of the
varying elliptic curve $\E$ and its $\myl$-torsion, in other words
from the moduli problem that is parametrized by the modular curve
$X(\myl)$.  This is the ``moduli  interpretation'' referred to in our
title; we give a uniform moduli interpretation of Eisenstein series in all
weights.  

In Section~\ref{section3} below, we in fact obtain
a family of ``moduli-friendly'' modular forms on $\Gammal$ over a more
general base field $\fieldk$, as coefficients in the Laurent
expansions of certain elliptic functions, which make sense
algebraically in the function field $\fieldk(\E)$.
We show that the modular forms we construct all belong
to a certain graded ring $\ringRl$ of modular forms on $\Gammal$, and
prove that the algebra $\ringRl$ is generated by the Eisenstein series
of weight~$1$ on $\Gammal$, when $\myl \geq 3$
(Theorem~\ref{theorem3.12}).  This result is similar to the
results proved in~\cite{BorisovGunnells}, where Borisov and Gunnells
define and study toric modular forms on $\Gamma_1(\myl)$, and
prove that the ring of toric modular forms is generated by certain 
Eisenstein series in weight~$1$, and that it is stable under the
Hecke operators $T_n$ for $\Gamma_1(\myl)$; their proofs rely
on $q$-expansions of modular forms.  Thus the results in this article
include a generalization to $\Gammal$ of the ring of toric modular
forms introduced in~\cite{BorisovGunnells}.
(See also~\cite{Cornelissen}, which studies the ring generated by
weight~$1$ Eisenstein series in the Drinfeld modular case.)  The above
article~\cite{BorisovGunnells}, as well as the subsequent
articles~\cite{BorisovGunnellsNonvanishing,BorisovGunnellsHigherWeight,
BorisovGunnellsPopescuEquations}, were a definite inspiration for
several of the results in this article, even though our proofs tend
to proceed along different lines (most notably, without any
$q$-expansions).

Sections \ref{section4} and~\ref{section5} contain the technical
heart of this article.
Continuing the analogy with~\cite{BorisovGunnells}, we also prove
that $\ringRl$ is stable under the Hecke algebra.
We first combine various relations between the modular forms in a
pleasantly intricate way to deduce
Hecke invariance in weights $2$ and~$3$
(Propositions \ref{proposition4.6}, \ref{proposition4.8},
and~\ref{proposition4.11}).  
Combining this result with analytic
techniques (Rankin-Selberg and nonvanishing of $L$-functions), along
with standard results on sufficiently positive line bundles on
curves, we prove that over $\C$, the ring $\ringRl$ contains
\emph{all} modular forms of weights $\wtj \geq 2$
(Theorem~\ref{theorem5.1}), and thus ``misses'' only the cusp forms in
weight~$1$.  This is the main result of our article.
Theorem~\ref{theorem5.1} is analogous to the results
in~\cite{BorisovGunnellsNonvanishing,BorisovGunnellsHigherWeight}
for toric modular forms on $\Gamma_1(\myl)$: the authors
prove there that the cuspidal part of the toric modular forms in
weight~$2$ consists of all cusp forms with nonvanishing central
$L$-value, while in weight $\wtj \geq 3$, the cuspidal part is all of
$\cuspforms_\wtj(\Gamma_1(\myl))$.  Their approach also uses
nonvanishing of $L$-functions, but is otherwise somewhat different.

We next apply Theorem~\ref{theorem5.1} to produce models of the modular
curve $X(\myl)$.  Our final result,
Theorem~\ref{theorem5.5}, can be stated in the following
striking manner: for $\myl \geq 3$, the slopes of lines joining the
$\myl$-torsion points of any \emph{one} elliptic curve over $\Q$
with $j \neq 0, 1728$ (for example, $\E^\zirc: y^2 = x^3 + 3141x + 5926$)
contain enough information to deduce equations for $X(\myl)$, which
parametrizes the $\myl$-torsion of \emph{all} elliptic curves.
We can find the equations for $X(\myl)$ using only exact computations
in the number field $\Q(\E^\zirc[\myl])$, and we obtain a model for
$X(\myl)$ over the cyclotomic 
field $\Q(\boldmu_\myl)$.  No infinite series or other
approximations are involved.  The model we obtain for $X(\myl)$
is in the form called ``Representation~B'' in~\cite{KKMasymptotic}
(where we use it to compute efficiently in the Jacobian of $X(\myl)$).
The idea is that the different level structures on $\E^\zirc$ give
rise to many points on $X(\myl)$ (embedded projectively via
$\modforms_2(\Gammal)$), and that only one curve can be reasonably
interpolated through these points. 

Since our results are moduli-friendly and mostly algebraic
(except for the analytic Theorem~\ref{theorem5.1}), much of our
theory works for a general base field $\fieldk$,
provided $6\myl \neq 0$ in $\fieldk$, and $\fieldk$ contains
the $\myl$th roots of unity.
Our approach proceeds entirely via
moduli of elliptic curves, and never involves $q$-expansions.  We
hope that these ideas can generalize to
modular forms on indefinite quaternion algebras and Shimura
curves.

\textbf{Acknowledgements.}
This research was partially supported by the University Research Board at the
American University of Beirut, and the Lebanese National Council for
Scientific Research, through the grants ``Equations for modular and
Shimura curves''.  The author is grateful to L. Merel for helpful
discussions about the Hecke action, and to R. Ramakrishna for useful
comments on the manuscript.

\section{Eisenstein series and Laurent expansions of elliptic functions}
\label{section2}

Our first goal in this section is to reexpress the sum defining
Eisenstein series so that it converges absolutely for all weights
$\wtj \geq 1$, without the need for Hecke's method of analytic
continuation when $\wtj \leq 2$.
As usual, let $\tau \in \Half$, where $\Half$ is the 
complex upper half plane, and consider the lattice
$L = L_\tau = \Z + \Z\tau$.

\begin{definition}
\label{definition2.1}
For $\torsi, \torsj \in \Z$, let $\alpha = \alpha_\tau
  = (\torsi \tau + \torsj)/\myl \in \frac{1}{\myl}L_\tau$.
For an integer $\wtj \geq 1$ and $s \in \C$, recall,
following~\cite{HeckeEisenstein}, the Eisenstein series of weight~$\wtj$ on
the principal congruence subgroup $\Gammal$, where $\myl \geq 1$:
\begin{equation}
\label{equation2.1}
\begin{split}
\Ej&(\tau, \alpha; s)
   = \sum_{\omega \in L_\tau}{}'
       \frac{1}{(\alpha+\omega)^\wtj \abs{\alpha+\omega}^{2s}}
\\
   &= \sum_{(\intm, \intn) \in \Z^2}{}'
       \bigl[(\intm + \torsi/\myl)\tau + \intn + \torsj/\myl\bigr]^{-\wtj}
       \bigl|(\intm + \torsi/\myl)\tau + \intn + \torsj/\myl\bigr|^{-2s},
\\
\end{split}
\end{equation}
\begin{equation}
\label{equation2.2}
\Ej(\tau, \alpha) = \Ej(\tau, \alpha; 0),
                           \quad \text{by analytic continuation}.
\end{equation}
Here the notation $\sum_\omega'$ omits $\omega = -\alpha$ in case we have
$\alpha \in L_\tau$; similarly for $\sum_{(\intm, \intn)}'$.
If the series already converges absolutely for $s=0$, we can 
write directly
\begin{equation}
\label{equation2.2.2}
\Ej(\tau, \alpha) = \sum_{\omega \in L_\tau}{}' (\alpha + \omega)^{-\wtj},
\quad \text{for } \wtj \geq 3.
\end{equation}
In general, for 
$\wtj \geq 1$, Hecke showed that $\Ej(\tau,\alpha;s)$ can be
analytically continued to all $s \in \C$, and that
$\Eone(\tau,\alpha)$ is a holomorphic function of $\tau$, while
$\Etwo(\tau, \alpha)$ is the sum of 
$-2\pi i / ( \tau - \overline{\tau})$ and a holomorphic function of $\tau$.
\end{definition}

The parameter $\alpha$ corresponds to a point
$\pointP_\alpha$ in the $\myl$-torsion of the elliptic curve
$\E = \E_\tau = \C/L_\tau$.  We can also define the $\Ej$ for a divisor,
either of numbers $\alpha \in \C$ or of points $\pointP \in \E$.  
We introduce the following notation to distinguish sums in the
additive groups $\C$ and $\E$ from the formal sums of points in
divisors.  
\begin{itemize}
\item
A divisor on $\C$ will be written $\tilde{\divD} = \sum_{\alpha} m_\alpha
(\alpha)$, and its image in $\E$ is $\divD =
\sum_{\alpha} m_\alpha 
(\pointP_\alpha)$.  (Here $m_\alpha \in \Z$.)  The $\alpha$ need not
be distinct modulo $L$, so cancellation can occur in the formal sum
for $\divD$.  We call $\tilde{\divD}$ a \emph{lift} of $\divD$.
\item
We denote by $\Pzero \in \E$ the additive identity in that group.
\item
The group operations of addition, inversion, and multiplication by an
integer $n \in \Z$ on points $\pointP, \pointQ \in \E$ are given by
\begin{equation}
\label{equation2.3}
\pointP, \pointQ \mapsto \pointP \oplus \pointQ,
\qquad
\pointP \mapsto \ominus \pointP = [-1]\pointP,
\qquad
\pointP \mapsto [n] \pointP = \pointP \oplus \cdots \oplus \pointP.
\end{equation}
\end{itemize}

\begin{definition}
\label{definition2.2}
Let $\divD$ be a divisor on $\E$ that is supported on the $\ell$-torsion
points $\El$, and choose any lift
$\tilde{\divD} = \sum_\alpha m_\alpha (\alpha)$ of $\divD$ to $\C$. 
We then define the following Eisenstein series on $\Gammal$,
depending linearly on $\divD$:
\begin{equation}
\label{equation2.4}
\Ej(\tau, \divD; s) = \sum_\alpha m_\alpha \Ej(\tau, \alpha; s),
\qquad\qquad
\Ej(\tau, \divD) = \Ej(\tau, \divD; 0).
\end{equation}
It is immediate that the definition does not depend on the choice of lift
$\tilde{\divD}$.
We remind the reader that the values
$\alpha \in \frac{1}{\myl} L_\tau$
(and corresponding points $\pointP_\alpha \in \El$) 
vary with $\tau$, as in Definition~\ref{definition2.1}.
\end{definition}

Our observation is that suitable choices of the lift $\tilde{\divD}$ lead
to series for $\Ej(\tau, \divD; s)$ with good convergence for all
$\wtj \geq 1$.  To motivate this, recall that a divisor 
$\divD = \sum_\alpha m_\alpha (\pointP_\alpha)$ on $\E$ is principal, of
the form $\divD = \Divisor(\thef)$ for some meromorphic function
$\thef$ on $\E$, if and only if
\begin{equation}
\label{equation2.5}
\deg \divD := \sum_\alpha m_\alpha = 0, \qquad\qquad
\bigoplus \divD := \bigoplus_\alpha [m_\alpha] \pointP_\alpha = \Pzero.
\end{equation}
The second sum above is evaluated in the group $\E$.

\begin{definition}
\label{definition2.3}
Let $\divD$ be a principal divisor on $\E$.  A \emph{principal lift} of
$\divD$ is a divisor $\tilde{\divD} = \sum_\alpha m_\alpha (\alpha)$ on
$\C$ satisfying
\begin{equation}
\label{equation2.6}
\sum_\alpha m_\alpha = 0, \qquad\qquad
\sum_\alpha m_\alpha \alpha = 0
        \quad \text{ (both sums evaluated in } \C). 
\end{equation}
An arbitrary lift $\tilde{\divD}$ would \textit{a priori} merely satisfy
$\sum_\alpha m_\alpha \alpha \in L$.
\end{definition}

Principal lifts always exist.  For example, let
$\alpha = (\torsi\tau + \torsj)/\myl$, and take the principal divisor
$\divD = \myl(\pointP_\alpha) - \myl(\Pzero)$.  Then one possible principal
lift of $\divD$ is
\begin{equation}
\label{equation2.7}
\tilde{\divD} = (\myl + 1)(\alpha) - (\alpha + \torsi \tau + \torsj)
   - \myl(0).
\end{equation}

\begin{proposition}
\label{proposition2.4}
Given a principal divisor $\divD$ supported on $\El$, choose a principal
lift $\tilde{\divD}$ as in~\eqref{equation2.6}.  Then
\begin{equation}
\label{equation2.8}
\sum_{\alpha}
    \frac{m_\alpha}{(\alpha + \omega)^\wtj \abs{\alpha + \omega}^{2s}}
 = O \left( \frac{1}{\abs{\omega}^{2s+j+2}} \right),
   \text{ for large } \abs{\omega}.
\end{equation}
We hence obtain for all $\wtj \geq 1$ the following convergent double
series (where the notation $\sum_\alpha'$ means that we omit $\alpha =
-\omega$ if it appears in the inner sum):
\begin{equation}
\label{equation2.9}
\Ej(\tau, \divD)
  = \left.
      \sum_{\omega \in L} \sum_{\alpha}{}'
        \frac{m_\alpha}{(\alpha + \omega)^\wtj \abs{\alpha + \omega}^{2s}}
    \right|_{s=0}
 =  \sum_{\omega \in L}
       \left( \sum_{\alpha}{}' \frac{m_\alpha}{(\alpha + \omega)^\wtj}
       \right).   
\end{equation}
Note that the outer sum over $\omega$ is absolutely convergent for
$\Real s > -\wtj/2$, even though the double sum converges only
conditionally.
\end{proposition}
\begin{proof}
The following expansion follows from Taylor's theorem, or from the
binomial series for $(1+\alpha/\omega)^{-\wtj}
\abs{1+\alpha/\omega}^{-2s} =
(1+\alpha/\omega)^{-s-\wtj}(1+\overline{\alpha}/\overline{\omega})^{-s}$:
\begin{equation}
\label{equation2.11}
\frac{1}{(\alpha+\omega)^\wtj \abs{\alpha+\omega}^{2s}} 
 = \frac{1}{\omega^\wtj \abs{\omega}^{2s}}
   - \frac{(s+\wtj)\alpha}{\omega^{\wtj+1}\abs{\omega}^{2s}}
   - \frac{s\>\overline{\alpha}}{\omega^{\wtj-1}\abs{\omega}^{2s+2}}
   + O\left( \frac{1}{\abs{\omega}^{2s+\wtj+2}} \right).
\end{equation}
The estimate holds for $\abs{\omega} > 2\abs{\alpha}$,
with an implied constant in the $O(\cdot)$ that depends on $\alpha$,
$\wtj$, and $s$, and is uniform in $\tau$ when $\tau$ is restricted to
a compact subset of $\Half$.  Our result now follows by
multiplying~\eqref{equation2.11} by $m_\alpha$ and summing over $\alpha$.
\end{proof}

\begin{remark}
\label{remark2.5}
Note that we always obtain holomorphic functions of $\tau$ above.  In
the setting of weight~$\wtj=2$, this arises because we have always
taken $\deg \divD = 0$, so the nonholomorphic terms cancel.
\end{remark}

Proposition~\ref{proposition2.4} allows us to rederive Hecke's second
definition of weight~$1$ Eisenstein series as ``division values'' of the
Weierstrass $\zeta$ function in Section~6 of~\cite{HeckeZurTheorie},
as well as Corollary~3.4.24  of~\cite{Katz}; we reprove those results
in~\eqref{equation2.14} below.  Recall the series for $\zeta(z)$:
\begin{equation}
\label{equation2.12}
\zeta(z)
  = \frac{1}{z} +
    \sum_{0 \neq \omega \in L} \left[
         \frac{1}{z - \omega} + \frac{1}{\omega} + \frac{z}{\omega^2}
         \right]
  = \frac{1}{z} +
    \sum_{0 \neq \omega \in L} \frac{z^2}{(z-\omega)\omega^2}.
\end{equation}
It is a standard fact that
$\zeta(z + \intm \tau + \intn) = \zeta(z) + 2\intm \eta_2 + 2\intn \eta_1$
for $\intm, \intn \in \Z$ (with ``constants'' $\eta_i = \eta_i(L)$
satisfying $2\eta_1\tau - 2\eta_2 = 2\pi i$).  Here
we follow the notation of Chapter~IV of \cite{Chandrasekharan}; note that
Hecke and other authors use $\eta_i$ for what we have called $2\eta_i$.
Moreover, $\zeta$ is an odd function of $z$, and in fact its
Laurent expansion near $0$ is $\zeta(z) = z^{-1} + O(z^3)$.

\begin{corollary}
\label{corollary2.6}
Let $\divD$ be a principal divisor supported on $\El$, and take a
principal lift $\tilde{\divD} = \sum_\alpha m_\alpha (\alpha)$ for
which every instance of $\Pzero$ in $\divD$ is lifted to $\alpha =
0$.  Then 
\begin{equation}
\label{equation2.13}
\Eone(\tau,\divD) = \sum_{\alpha \neq 0} m_\alpha \zeta(\alpha).
\end{equation}
Moreover, let $\pointP_\alpha \in \El - \{\Pzero\}$, with any choice
of lift
$\alpha = (\torsi \tau + \torsj)/\myl$ with $\torsi, \torsj \in \Z$.
Then
\begin{equation}
\label{equation2.14}
\begin{split}
\Eone(\tau,\pointP_\alpha)
         &=
     \zeta(\alpha)
   + \frac{1}{\myl} \left[
         \zeta(\alpha) - \zeta(\alpha + \torsi \tau + \torsj)  
     \right] \\
         &=
     \zeta\left(\frac{\torsi \tau + \torsj}{\myl}\right)
   - \frac{\torsi}{\myl} \cdot 2\eta_2
   - \frac{\torsj}{\myl} \cdot 2\eta_1.\\
\end{split}
\end{equation}
\end{corollary}
\begin{proof}
Write
$\tilde{\divD} = m_0 (0) + \sum_{\alpha \neq 0} m_\alpha (\alpha)$,
with $\alpha \neq 0 \implies \alpha \notin L$ by our assumption on
$\tilde{\divD}$.  Changing the sign of $\omega$
in~\eqref{equation2.12}, we obtain
\begin{equation}
\label{equation2.15}
\sum_{\alpha \neq 0} m_\alpha \zeta(\alpha)
 = \sum_{\alpha \neq 0} \frac{m_\alpha}{\alpha}
 + \sum_{\omega \neq 0}
   \sum_{\alpha \neq 0}
      \left[
         \frac{m_\alpha}{\alpha + \omega} 
        - \frac{m_\alpha}{\omega}
        + \frac{m_\alpha \alpha}{\omega^2}
      \right].
\end{equation}
The change of order of summation is justified by the good convergence
of the series for $\zeta$ and because the sum over $\alpha$ is finite.
Since $\tilde{\divD}$ satisfies~\eqref{equation2.6}, we have
$\sum_{\alpha \neq 0} m_\alpha = - m_0$ and $\sum_{\alpha \neq 0}
m_\alpha \alpha = 0$, which allows us to rewrite the above sum in the
form of~\eqref{equation2.9} (at the cost of replacing absolute
convergence with conditional convergence), and hence to
obtain~\eqref{equation2.13}.  Now apply this result in the case
$\divD = \myl (\alpha) - \myl(0)$, using the principal lift
$\tilde{\divD}$ from~\eqref{equation2.7}.  This
yields~\eqref{equation2.14}, because
$\Eone(\tau, \myl (\alpha) - \myl(0))
= \myl \Eone(\tau, \alpha) - \myl \Eone(\tau, 0)$ and
$\Eone(\tau,0) = 0$ (more generally, $\Ej(\tau, -\beta; s) = (-1)^\wtj
\Ej(\tau, \beta; s)$).
\end{proof}

We now turn to the second goal of this section, which is to express
Eisenstein series on $\Gammal$ as coefficients in Laurent expansions of
certain elliptic functions.  The reader is also referred to~\cite{Pasol}
for some related results from a different viewpoint.

\begin{definition}
\label{definition2.7}
Let $\divD$ be a principal divisor on $\E$, and let $m_0$ be the
multiplicity of $\Pzero$ in $\divD$.  We define an element
$\fD$ of the function field of $\E$ by the requirements
\begin{equation}
\label{equation2.16}
\Divisor(\fD) = \divD, \qquad\qquad
\fD = z^{m_0}
      ( 1 + O(z) ),\quad \text{ near } z=0.
\end{equation}
Here the first requirement determines $\fD$ up to a nonzero constant
factor, and the second requirement (viewing $\fD$ as an
elliptic function on $\C$ with respect to $L$) normalizes the constant
so as to fix our choice of $\fD$.  Our normalization ensures that for
principal divisors $\divD$ and $\divE$, 
\begin{equation}
\label{equation2.16.5}
  \thef_{\divD + \divE} = \fD \cdot \fE.
\end{equation}
\end{definition}
The precise normalization of the constant factor in
$\fD$ will be needed in later sections of this article; it is not
essential in this section, where we mainly consider the
logarithmic differential $d\fD/\fD$.

\begin{theorem}
\label{theorem2.8}
Let $\divD$ be a principal divisor, and take a principal lift
$\tilde{\divD} = \sum_\alpha m_\alpha (\alpha)$.   Make the same
assumption on $\tilde{\divD}$ as in Corollary~\ref{corollary2.6}.
Then
\begin{equation}
\label{equation2.17}
\frac{d\fD}{\fD} = \sum_{\alpha}  m_\alpha \zeta(z-\alpha) \, dz
 = \sum_{\omega \in L} \left[
      \sum_{\alpha} \frac{m_\alpha}{z - \alpha - \omega}
   \right] dz,
\end{equation}
where the last series has similar convergence properties to the series
of~\eqref{equation2.9}.  Furthermore, if $\divD$ is supported on
$\El$, then the Laurent series expansion of $d\fD/\fD$ near $z=0$ is
\begin{equation}
\label{equation2.18}
\frac{d\fD}{\fD} = \left(
        \frac{m_0}{z} - \sum_{\wtj \geq 1} \Ej(\tau,\divD) z^{\wtj-1}
        \right) dz.
\end{equation}
\end{theorem}
\begin{proof}
It is classical (see, for example, Section IV.3 of \cite{Chandrasekharan})
that we can express $\fD$ up to a nonzero constant $C = C_\tau$ in
terms of the Weierstrass $\sigma$ function, provided that we have
taken a principal lift $\tilde{\divD}$:
\begin{equation}
\label{equation2.19}
\fD(z) = C \prod_{\alpha} \Bigl[\sigma(z - \alpha)^{m_\alpha}\Bigr].
\end{equation}
Taking logarithmic differentials yields the first equality
in~\eqref{equation2.17}, since $\sigma'/\sigma = \zeta$.  The second
equality now follows from substituting the series for $\zeta$ and
using the fact that
$\sum_\alpha [m_\alpha/\omega + m_\alpha(z-\alpha)/\omega^2] = 0$. 

We can now prove~\eqref{equation2.18}.  The first term in the Laurent
expansion is easy, and the other terms are equivalent to showing
that $\Res_{z=0} \left[ z^{-\wtj} \frac{d\fD}{\fD} \right] = 
-\Ej(\tau,\divD)$ for $\wtj \geq 1$.  This residue can be computed by a
contour integral on a small circle enclosing $z=0$.   Since the sum over
$\omega$ in~\eqref{equation2.17} converges well, we are justified in
computing the residue term-by-term, using the expansion
$\frac{1}{z-\beta}
 = - \frac{1}{\beta} - \frac{z}{\beta^2} - \frac{z^2}{\beta^3}
      - \cdots$
for $\beta \neq 0$ to compute residues for each inner sum over $\alpha$
that occurs as a term in the sum over $\omega$.  Comparing
with~\eqref{equation2.9} yields the desired result. 
\end{proof}

\begin{remark}
\label{remark2.9}
For $\wtj \geq 2$, one can give a more classical proof that the
coefficient of $z^\wtj$ in $d\fD/\fD$ is $-\Ej(\tau,\divD)$, 
by taking the contour 
integral of $z^{-\wtj}\frac{d\fD}{\fD}$ around a large parallelogram with
center at $0$ and sides tending to infinity.
\end{remark}

The above theorem appears to relate Laurent expansions of elliptic
functions only to those Eisenstein series $\Ej(\tau,\divD)$ where
$\divD$ is principal.  On the other hand, $\Ej(\tau,\divD)$ depends
linearly on $\divD$ (as does $d\fD/\fD$,
by~\eqref{equation2.16.5}), so we are led to consider linear
combinations of Eisenstein series.

\begin{proposition}
\label{proposition2.10}
Let $\myl \geq 2$.  Then for all $\wtj \geq 1$, the span of the
Eisenstein series  
$\{ \Ej(\tau,\divD) \mid \divD \text{ principal, supported on } \El\}$
consists of all holomorphic Eisenstein series of weight~$\wtj$ on
$\Gammal$. 
\end{proposition}
\begin{proof}
The Eisenstein series for the principal divisors
$\{\myl (\pointP) - \myl(\Pzero) \mid \pointP \in \El\}$ and
$\divD = \bigl[\sum_{\pointP \in \El} (\pointP)\bigr] - \myl^2 (\Pzero)$
are $\{\myl\Ej(\tau,\pointP) - \myl\Ej(\tau, \Pzero) \mid \pointP \in \El\}$
and $\Ej(\tau, \divD) =(\myl^\wtj - \myl^2)\Ej(\tau, \Pzero)$.  Their span
includes all the $\Ej(\tau,\pointP)$, as desired, except when $\wtj = 2$.
However, in the case $\wtj=2$, the holomorphic Eisenstein series are
spanned precisely by the $\Ej(\tau,\pointP) - \Ej(\tau, \Pzero)$, since we
want the nonholomorphic terms $-2\pi i / ( \tau - \overline{\tau})$ to
cancel.
\end{proof}

\begin{remark}
\label{remark2.11}
It is convenient to allow general $\divD$ that are supported on $\El$.
Then $\divD$ may not have degree zero, let alone be principal.  However,
the divisor $\divD_z := \divD - (\deg\divD) (\Pzero)$ has degree zero, and
$\myl \divD_z$ is principal.  Hence we can formally define
$\fD = (\thef_{\myl \divD_z})^{1/\myl}$ for compatibility
with~\eqref{equation2.16.5}.  Note that if $\bigoplus \divD \neq
\Pzero$, then $\fD$ cannot be an elliptic function with respect to
$L$; its formal logarithmic derivative is nonetheless always periodic 
with respect to $L$, and we can simply take $d\fD/\fD =
(1/\myl)d\thef_{\myl \divD_z}/\thef_{\myl \divD_z}$ as a definition.  With
this convention, \eqref{equation2.18} continues to hold (with coefficients
$\Ej(\tau,\divD_z)$), and we can
obtain an analog of~\eqref{equation2.17} as a series with good
convergence properties, similarly to our derivation
of~\eqref{equation2.14}.

Looking a bit further, we note that $\thef_{\myl \divD_z}$ has
zeros and poles with multiplicity everywhere divisible by $\myl$.
Hence $\fD$ makes sense as a meromorphic function on $\C$.  We use
this to normalize the choice of $\myl$th root $\fD$ as
in~\eqref{equation2.16}, so that its Laurent series begins with
$z^{m_0-\deg \divD}$, since $m_0 - \deg \divD$ is the multiplicity of
$\Pzero$ in $\divD_z$.
Then these general $\fD$ are products of (positive and negative) powers of
the $\fP = (\thef_{\myl(\pointP) - \myl(\Pzero)})^{1/\myl}$,
for $\pointP \in \El - \{\Pzero\}$.  For such a ``basic'' $\fP$,
Theorem~\ref{theorem2.8} then states that
\begin{equation}
\label{equation2.19.5}
\begin{split}
\frac{d\fP}{\fP}
 &= z^{-1}\Bigl(
  -1 - \sum_{\wtj \geq 1}
         \bigl(\Ej(\tau,\pointP) - \Ej(\tau,\Pzero)\bigr) z^{\wtj}
   \Bigr) dz\\
 &= z^{-1}\Bigl(-1 - \Eone(\tau,\pointP)z
              + \bigl(-\Etwo(\tau,\pointP) + \Etwo(\tau,\Pzero)\bigr)z^2
              + \cdots
          \Bigr)dz.
\end{split}
\end{equation}
Note that $\thef_{\Pzero} = 1$, so~\eqref{equation2.19.5} does not
quite hold for $\pointP = \Pzero$; the first coefficient in the
Laurent expansion becomes $0$ instead of $-1$.

The functions $\fP$ above are still elliptic functions, however with
respect to  the sublattice $\myl L$ of $L$.  The behavior of $\fP$ under
translations by $L$ is described by 
a Weil pairing; see Definition~\ref{definition4.1} in
Section~\ref{section4} below, where we work instead with the function
$\gP(z) = \fP(\myl z)$, which is elliptic with respect to the full
lattice $L$.
The approach of working with $\fP$ that are periodic with respect to
$\myl L$ is used in the work of Borisov and Gunnells on toric modular
forms~\cite{BorisovGunnells}.  They use the function $\vartheta =
\vartheta_{11}$ to write down what amounts to the same
function as $\fP$ when $\pointP = a/\myl + L$ is in the subgroup of
$\El$ generated by $\pointP_{1/\myl}$.  They then use the expansion of
$d\fP/\fP$ at $z = 0$ to define their toric modular forms
$s_{a/\myl}^{(k)}$ (see Section~4.4 of~\cite{BorisovGunnells}).  Thus
their $s_{a/\myl}^{(k)}$ are the same as our $\Ej(\tau,\divD)$
for the divisor $\divD = [a]\pointP_{1/\myl} - \Pzero$.  In particular,
the $s_{a/\myl}^{(k)}$ are Eisenstein series 
with respect to the larger group $\Gamma_1(\myl)$;
Borisov and Gunnells recognize this from the $q$-expansions, while our
approach is more direct.  An advantage of working with $\Gammal$ is
that we obtain the full space of holomorphic
Eisenstein series
in all weights, by
Proposition~\ref{proposition2.10}; see also Theorems \ref{theorem3.9}
and~\ref{theorem3.12} below.  In contrast, the ring of toric
modular forms on $\Gamma_1(\myl)$ does not always contain all
Eisenstein series on that group: see Remark~4.13
of~\cite{BorisovGunnellsNonvanishing}.
\end{remark}

\begin{remark}
\label{remark2.12}  
One can find the Laurent expansion of $\fD$ by 
formally exponentiating the integral of $d\fD/\fD$.  Keeping track of
the algebra, one obtains that $\fD$ has an expansion of the following
form near $z=0$:
\begin{equation}
\label{equation2.20}
\fD = z^{m_0}
      ( 1 + F_1(\tau) z + F_2(\tau) z^2 + \cdots ),
\end{equation}
where $F_\wtj$ is a modular form on $\Gammal$ of weight~$\wtj$, expressible
as a polynomial in the $\Ej(\tau,\divD)$.  This approach is used
extensively in~\cite{BorisovGunnells}.  In the next section, we study the
Laurent series of $\fD$ directly in a purely algebraic setting over a more
general field $\fieldk$, and reformulate and extend the results of this
section algebraically.  For now, we simply note the result for $\fP$,
obtained from~\eqref{equation2.19.5}:
\begin{equation}
\label{equation2.21}
\fP = z^{-1}\left[ 1
               - \Eone z
               + \frac{(\Eone^2 - \tilde{\Etwo})}{2} z^2
               - \left( \frac{\Ethree}{3}
                      - \frac{\Eone \tilde{\Etwo}}{2}
                      + \frac{\Eone^3}{6}
                  \right) z^3
           + \cdots \right]
\end{equation}
where we wrote $\Eone = \Eone(\tau,\pointP)$, $\tilde{\Etwo} =
\Etwo(\tau,\pointP) - \Etwo(\tau,\Pzero)$, and 
$\Ethree = \Ethree(\tau,\pointP)$
to save space.
%
%
\end{remark}

\section{Algebraic reformulation and the ring $\ringRl$ of modular
forms}
\label{section3}

Our first step in ``algebrizing'' the results of the previous section is
to normalize the equation of our elliptic curve $\E$.  We embed $\E$ into
the projective plane $\Projective^2$ as follows (note the factor $1/2$):
\begin{equation}
\label{equation3.1}
z \mapsto \pointP_z = [\wp(z;L): (1/2) \wp'(z;L) : 1] = [x(z): y(z): 1].
\end{equation}
As usual, $\Pzero = [0:1:0]$ is the identity element.  The affine algebraic
equation of $\E$ and the invariant differential $\omega$ on $\E$ are
\begin{equation}
\label{equation3.2}
\E: y^2 = x^3 + \acoeff x + \bcoeff,  \qquad \omega = dx/(2y) = dz.
\end{equation}
Here $\acoeff = \acoeff(\tau)$ and $\bcoeff = \bcoeff(\tau)$ are, up to
constant factors, the Eisenstein series of level~$1$ and weights $4$
and~$6$, respectively:
\begin{equation}
\label{equation3.3}
\acoeff(\tau) = -15 \Efour(\tau,0)
              = -15 \sum_{0 \neq \omega \in L_\tau} \omega^{-4},
\qquad
\bcoeff(\tau) = -35 \Esix(\tau,0).
\end{equation}
The symbol $\omega$ in~\eqref{equation3.3} denotes an element of $L$, but
for the rest of this article it will refer almost exclusively to the
invariant differential, as in~\eqref{equation3.2}.

We now regard the
family $\{\E_\tau \mid \tau \in \Half\}$ as a single elliptic curve $\E$
over the rational function field $\C(\acoeff, \bcoeff)$ in two independent
transcendental variables.  We can work with more general fields $\fieldk$
instead of $\C$; in that case, $\E$ is a curve over the field 
$\fieldbigK = \fieldk(\acoeff,\bcoeff)$. 
Since we wish to use Weierstrass
normal form for $\E$, and also need to consider the $\myl$-torsion
throughout, we require $6\myl$ to be invertible in $\fieldk$, and for
$\fieldk$ to contain the group $\boldmu_\myl$ of $\myl$th roots of unity
(so as to accommodate the Weil pairing later).  We work over the
$\myl$-torsion extension field $\fieldbigKl$ of $\fieldbigK$:
\begin{equation}
\label{equation3.4}
\fieldbigKl = \fieldbigK(\El) 
  = \fieldk\bigl(\acoeff,\bcoeff,\bigl\{\xP, \yP \mid
            \pointP = (\xP,\yP)
                 \in \El(\fieldbigKbar) - \{\Pzero\}
              \bigr\}\bigr).
\end{equation}
The field $\fieldbigKl$ contains algebraic analogs of several (in fact,
all) complex modular 
forms on $\Gammal$.  Besides $\acoeff$ and $\bcoeff$ in weights $4$
and~$6$, which generate the algebraic analog of the
graded ring of modular forms on $\Gamma(1)$, the coordinates $\xP$ and
$\yP$ of points $\pointP \in \El - \{\Pzero\}$ are obvious
analogs of holomorphic Eisenstein series of weights $2$ and~$3$.
Specifically, over $\C$, let $\pointP = \pointP_\alpha$
for $\alpha = \alpha_\tau \in \frac{1}{\myl}L_\tau - L_\tau$.  Then the
usual series for $\wp$ and $\wp'$, along with~\eqref{equation2.9},
immediately give us
\begin{equation}
\label{equation3.5}
\xP = \wp(\alpha; L_\tau) = \Etwo(\tau, \alpha) - \Etwo(\tau, 0),
\qquad
\yP = (1/2)\wp'(\alpha; L_\tau) = -\Ethree(\tau,\alpha).
\end{equation}
\begin{remark}
\label{remark3.4}
The weights of the algebraic analogs of modular forms we list in
$\fieldbigKl$ can be defined 
intrinsically by considering, for each $\scalarlambda \in \fieldk^\times$, 
the automorphism of $\fieldbigKl$ and corresponding isomorphism of
elliptic curves given by:
\begin{equation}
\label{equation3.12}
\begin{split}
& \qquad\qquad\quad
\acoeff \mapsto \scalarlambda^4 \acoeff, \qquad\qquad
\bcoeff \mapsto \scalarlambda^6 \bcoeff, \qquad\qquad
\omega \mapsto \scalarlambda^{-1} \omega,\\
&(x,y) \in \E: y^2 = x^3 + \acoeff x + \bcoeff 
\mapsto (\scalarlambda^2 x, \scalarlambda^3 y)
   \in \E': y^2 = x^3 + \scalarlambda^4\acoeff x + \scalarlambda^6\bcoeff.\\
\end{split}
\end{equation}
This automorphism naturally sends $\xP \mapsto \scalarlambda^2 \xP$
and $\yP \mapsto \scalarlambda^3 \yP$.
\end{remark}

We now introduce, for certain divisors $\divD$, another collection
$\{\lambdaD\}$ of algebraic analogs of modular forms of weight~$1$,
which we later show to be Eisenstein series. 
\begin{definition}
\label{definition3.0.1}
Consider a divisor on $\E$ of the form
$\divD = (\pointP) + (\pointQ) + (\pointR)- 3(\Pzero)$,
where $\pointP, \pointQ, \pointR \in \El - \{\Pzero\}$ satisfy
$\pointP \oplus \pointQ \oplus \pointR = \Pzero$; thus the divisor $\divD$
is principal, and the points $\pointP, \pointQ, \pointR$ are
collinear in the affine Weierstrass model of $\E$.  Write
the equation of the line joining these three points as
$y = \lambdaD x + \nuD$.  We have thus defined
$\lambdaD$ to be the slope of the line through these three points.
We also define $\thelambda_{(\pointP) + (\pointQ) + (\pointR)} =
\lambdaD$; i.e., the $\thelambda$ notation ignores the $\Pzero$ terms
in $\divD$.

We similarly define
$\thenu_{(\pointP) + (\pointQ) + (\pointR)} = \nuD$ to be the
$y$-intercept of the line.
\end{definition}

\begin{proposition}
\label{proposition3.0.2}
The slope $\lambdaD$ in Definition~\ref{definition3.0.1} is the
algebraic analog of a weight~$1$ modular form on~$\Gammal$.
Similarly, $\nuD$ is a weight~$3$ modular form.
\end{proposition}
\begin{proof}
This follows easily from~\eqref{equation1.1}, which, when combined
with~\eqref{equation2.13}, shows that $\lambdaD$ is in fact an 
Eisenstein series; the statement about $\nuD$ follows because $\nuD =
\yP - \lambdaD \xP$.  We prefer however to give a different
self-contained proof that $\lambdaD$ is modular.  We then complete the
proof that $\lambdaD$ is an Eisenstein series in
Corollary~\ref{corollary3.6}.

Our direct proof for $\lambdaD$ proceeds from the equations
\begin{equation}
\label{equation3.5.2}
\lambdaD =
\begin{cases}
(\yP-\yQ)/(\xP-\xQ), &\text{if } \pointP \neq \pointQ,\\
(3\xP^2 + \acoeff)/2\yP, &\text{if } \pointP = \pointQ,\\
\end{cases}
\end{equation}
from which it follows that, in case $\fieldk = \C$, 
the value $\lambdaD$ (viewed as a function of $\tau$) is a ratio of
modular forms of suitable weights, and hence transforms under $\Gammal$
like a modular form of weight~$1$. (We pause to note that the denominators
above are not zero: for example, if $\xP = \xQ$ and $\pointP \neq \pointQ$,
then we must have $\pointP = \ominus \pointQ$, contradicting $\pointR \neq
\Pzero$.) 
However, the quotient expression for $\lambdaD$ might
have poles on $\Half$ or at the cusps.  The formulas for the
addition law on $\E$, plus~\eqref{equation3.5}, yield
\begin{equation}
\label{equation3.5.4}
\lambdaD^2 = \xP + \xQ + \xR = \Etwo(\tau,\divD),
\end{equation}
a holomorphic form of weight~$2$.  This shows that $\lambdaD$ is
holomorphic.
\end{proof}

We now define a graded subring $\ringRl$ of $\fieldbigKl$ that will
feature prominently in our discussion.  Over $\C$, the ring $\ringRl$ 
will be a subalgebra of the ring of modular forms over
$\Gammal$, graded by weight.
\begin{definition}
\label{definition3.0.4}
If $\myl = 1$, define
\begin{equation}
\label{equation3.3.5}
\ringR_1 = \fieldk[\acoeff,\bcoeff]
\end{equation}
and if $\myl \geq 2$, define
$\ringRl$ to be the graded $\fieldk$-algebra generated by:
\begin{itemize}
\item The forms $\acoeff$ and $\bcoeff$, in weights $4$ and $6$,
\item All coordinates $\xP$, $\yP$, in weights $2$ and $3$, for $\pointP
  \in \El - \{\Pzero\}$,
\item All slopes $\lambdaD$, in weight~$1$, for divisors $\divD$ as in
  Definition~\ref{definition3.0.1}. 
\end{itemize}
Note that $\nuD = \yP - \lambdaD \xP \in \ringRl$, and that
$\ringR_{\myl'} \subset \ringRl$ for $\myl' | \myl$
(including $\myl' = 1$).
\end{definition}

Our first main result in this section is that for $\myl \geq 3$, the ring
$\ringRl$ is in fact generated by its elements of weight~$1$ (i.e., as we
shall see, by Eisenstein series of weight~$1$ on $\Gammal$).

\begin{theorem}
\label{theorem3.12}  
Assume that $\myl \geq 3$.
Then $\ringRl$ is generated by the $\lambdaD$, for $\divD$ as in Definition~\ref{definition3.0.1}.
\end{theorem}

\begin{remark}
\label{remark3.12.5}
When $\myl = 1$, $\ringR_1$ is of course generated by
$\acoeff,\bcoeff$.  When $\myl = 2$, write as usual $\E[2] = 
\{\Pzero, \pointP_1, \pointP_2, \pointP_3\}$ with $\pointP_i = (e_i,0)$ for
$1 \leq i \leq 3$.  Hence $x_{\pointP_i} = e_i$ and $y_{\pointP_i} = 0$ for
$1 \leq i \leq 3$, and all the $\lambdaD=0$ in this case; moreover,
$(x-e_1)(x-e_2)(x-e_3) = x^3 + \acoeff x + \bcoeff$.  We easily
obtain that $e_1$ and $e_2$ are algebraically independent, and that
$\ringR_2 = \fieldk[e_1, e_2]$ (note that $e_3 = -e_1 - e_2$).
Over $\C$, the ring $\ringR_2$ is the full ring of modular forms
on $\Gamma(2)$, and the generators $e_1,e_2$ are weight~$2$ Eisenstein
series.
\end{remark}

\begin{proof}[Proof of Theorem~\ref{theorem3.12}]
Let $\ringR' \subset \ringRl$ be the graded subalgebra generated
by all the $\lambdaD$.  Our goal is to show that the forms $\acoeff,
\bcoeff, \{\xP\}, \{\yP\}$, for $\pointP \in \El - \{\Pzero\}$,
all belong to $\ringR'$.
We begin by showing that all the $\{\xP\}$ belong to $\ringR'$.  This
boils down to a judicious use of~\eqref{equation3.5.4}, and
involves three cases, depending on $\myl$:
\begin{enumerate}
\item
If $\myl \geq 5$, let $\pointP$ be a point of exact order $\myl$, and
consider the following four elements of $\ringR'$ (note that
$x_{\ominus P} = \xP$):
\begin{equation}
\label{equation3.19}
\begin{split}
(\thelambda_{(\pointP) + (\pointP) + ([-2]\pointP)})^2 
       = \xP + \xP + x_{[-2]\pointP} &= 2\xP + x_{[2]\pointP}\\
(\thelambda_{(\pointP) + ([2]\pointP) + ([-3]\pointP)})^2
       = \xP + x_{[2]\pointP} + x_{[-3]\pointP} 
              &= \xP + x_{[2]\pointP} + x_{[3]\pointP}\\
(\thelambda_{(\pointP) + ([3]\pointP) + ([-4]\pointP)})^2
      \qquad\qquad\qquad\qquad
      &= \xP \qquad\quad + x_{[3]\pointP} + x_{[4]\pointP}\\
(\thelambda_{([2]\pointP) + ([2]\pointP) + ([-4]\pointP)})^2
      \qquad\qquad\qquad\qquad
      &= \qquad \quad 2x_{[2]\pointP} \qquad + x_{[4]\pointP}.\\
\end{split}
\end{equation}
Here the determinant
$\det \begin{pmatrix} 2&1&0&0\\ 1&1&1&0\\ 1&0&1&1\\0&2&0&1\\ \end{pmatrix}
= 6$ is invertible in $\fieldk$, and so each of $\xP, x_{[2]\pointP},
x_{[3]\pointP}, x_{[4]\pointP}$ can be expressed in terms of $\lambdaD$s,
hence belongs to $\ringR'$.  Now let $\pointP \in \El$ be a point of
order less than $\ell$.  We can find a basis $\{\pointQ, \pointR\}$
for $\El \isomorphic (\Z/\ell\Z)^2$, such that $\pointP = [d]\pointQ$
for some $d > 1$.
In that case, the points $\pointP' = (\ominus \pointP) \oplus \pointR =
[-d] \pointQ \oplus \pointR$ and $\pointP'' = \ominus \pointR$ both have
exact order $\myl$, so $x_{\pointP'}$ and $x_{\pointP''}$ both belong to
$\ringR'$.  The points $\pointP, \pointP', \pointP''$ are collinear, and so
$(\thelambda_{(\pointP)+(\pointP')+(\pointP'')})^2
     =\xP + x_{\pointP'} + x_{\pointP''}$ belongs to $\ringR'$, whence $\xP
\in \ringR'$.  (Alternatively, we can deal with the point $\pointP =
[d]\pointQ$ by using identities
analogous to~\eqref{equation3.19} to see that
$\xQ + x_{[n]\pointQ} + x_{[n+1]\pointQ} \in \ringR'$, and to deduce
inductively that the $x$-coordinates of all multiples $[n]\pointQ$ belong
to $\ringR'$ whenever $\pointQ$ has exact order $\myl$.)
\item
If $\myl = 3$, we simply note that
$(\thelambda_{3(\pointP)})^2  = 3\xP$ for all $\pointP \in
\E[3] - \{\Pzero\}$.
\item
If $\myl = 4$, let $\{\pointQ, \pointR\}$ be a basis for $\E[4] \isomorphic
(\Z/4\Z)^2$.  By the same technique as in the first case above, we see that
the following sums belong to $\ringR'$, being squares of suitable
$\thelambda$'s: 
\begin{equation}
\label{equation3.20}
\begin{matrix}
2\xQ & + x_{[2]\pointQ}, &&&&\\
&& 2\xR & + x_{[2]\pointR}, &&\\
\xQ && + \xR && + x_{\pointQ \oplus \pointR}, &\\
\xQ && + \xR &&& + x_{\pointQ \ominus \pointR}, \\
& x_{[2]\pointQ} &&& +  x_{\pointQ \oplus \pointR} &
                                 +  x_{\pointQ \ominus \pointR}, \\
&&& x_{[2]\pointR} & +  x_{\pointQ \oplus \pointR} &
                                 +  x_{\pointQ \ominus \pointR}.\\
\end{matrix}
\end{equation}
(For example, the fourth sum above is
$(\thelambda_{(\pointQ)+(\ominus\pointR)+(\pointR\ominus\pointQ)})^2$.)
The corresponding determinant is $-12$, again invertible, so we deduce in
particular that $\xQ, x_{[2]Q} \in \ringR'$.  Now any $\pointP \in \E[4] -
\{\Pzero\}$ has exact order either $4$ or~$2$.  So we can choose our basis
$\{\pointQ, \pointR\}$ so as to have $\pointP = \pointQ$ in the former
case, and $\pointP = [2]\pointQ$ in the latter case, thereby concluding
that $\xP \in \ringR'$.
\end{enumerate}
Now that we have shown that all the $\xP$ belong to $\ringR'$, let us show 
that all the $\yP$ also belong to $\ringR'$.  
Fix $\pointP \in \El - \{\Pzero\}$, and 
take any $\pointQ \in \El - \{\Pzero, \pointP, \ominus\pointP \}$.  Then
$(\yP - \yQ)/(\xP - \xQ)$ and $(\yP + \yQ)/(\xP - \xQ)$ are among our
$\thelambda$'s (the latter being the slope of the line through $\pointP$
and $\ominus \pointQ$), and so their sum $2\yP/(\xP - \xQ)$ belongs to
$\ringR'$.  Multiplying by $\xP - \xQ \in \ringR'$ shows that $\yP \in
\ringR'$.

Finally, take any $\pointP \in \El - \E[2]$.  Then $\acoeff = 
2\yP \thelambda_{(\pointP) + (\pointP) + ([-2]\pointP)} - 3\xP^2$
also belongs to
$\ringR'$, as does $\bcoeff = \yP^2 - \xP^3 - \acoeff \xP$.  
\end{proof}

\begin{remark}
\label{remark3.13}
We can also define a subring $\ringR'_A$ of $\ringR'$, corresponding to a
subgroup $A \subset \El$: let $\ringR'_A$ be generated by the forms
$\thelambda_{(\pointP) + (\pointQ) + (\pointR)}$, for $\pointP, \pointQ,
\pointR \in A - \{\Pzero\}$ with $\pointP \oplus \pointQ \oplus \pointR =
\Pzero$. 
Assume that $A \isomorphic \Z/m\Z \oplus \Z/\myl \Z$ with $m | \myl$ and
$\myl \geq 5$ (possibly $m = 1$).
Then our methods of proof show that $\acoeff, \bcoeff, \{\xP, \yP \mid
\pointP \in A - \{\Pzero\}\}$ all belong to $\ringR'_A$, as do the
appropriate $\thenu$'s coming from points in $A$.  Compare this to
Proposition~4.9 in~\cite{BorisovGunnells}.
\end{remark}

Our second main result in this section, Theorem~\ref{theorem3.9}
below, is the algebraic analog of Remark~\ref{remark2.12}.  We show
that the algebraic Laurent expansions of suitable elements $\fD$ of
the function field of $\E$ all have coefficients in $\ringRl$.
We take our Laurent expansions with respect to an algebraic
uniformizer $t$ at $\Pzero$:
\begin{equation}
\label{equation3.6}
t = -x/y \quad (= z - 2\acoeff z^5/5 + O(z^7) \text{ when } \fieldk = \C).
\end{equation}
Write $\Ohat$ for the completion of the local ring
of $\E$ over $\fieldbigKl$ at $\Pzero$; hence $\Ohat \isomorphic
\fieldbigKl[[t]]$ canonically, providing us with our expansions in terms
of $t$.  When $\fieldk$ has characteristic zero, we can still obtain the
analytic expansions in terms of $z$ from Section~\ref{section2}.  Indeed,
the analytic uniformizer $z$ still makes sense as an element of $\Ohat$,
since the relation $\omega = dz$ means that $z = \int \omega = t + 2a t^5/5
+ \cdots$, from~\eqref{equation3.7} below.

The meromorphic functions $x,y \in \fieldbigKl(\El)$ then have the
following algebraic Laurent expansions:
\begin{equation}
\label{equation3.7}
\begin{split}
x &= t^{-2} - \acoeff t^2 + \cdots
   = t^{-2}\bigl(1 - \acoeff t^4 + \cdots \bigr)
               \in t^{-2} \ringR_1[[t]],\\
-tx = y &= -t^{-3} + \acoeff t + \cdots
   = t^{-3}\bigl(-1 + \acoeff t^4 + \cdots \bigr),
                \in t^{-3} \ringR_1[[t]],\\
\omega &= (1 + 2\acoeff t^4 + \cdots) dt
                \in \ringR_1[[t]]dt.\\
\end{split}
\end{equation}
Moreover, the coefficient of $t^\wtj$ in the power series inside
each pair of parentheses above is always a weight~$\wtj$ homogeneous
element of the graded ring $\ringR_1$.
For all this, see for example Section~IV.1 
in~\cite{SilvermanI}; alternatively, one can start from the 
usual analytic expansion of $\wp$ in case $\fieldk = \C$ to obtain
expansions of $x$, $y$, and $t$ in terms of $z$.  Since $t = z +
O(z^5)$, we
obtain series for $z$, $x$, and $y$ in terms of $t$.

The form of the expansions in~\eqref{equation3.7} and the results of
Section~\ref{section2} suggest the following definition.
\begin{definition}
\label{definition3.7}
An $\ringRl$-balanced Laurent series in $t$ is a
series of the form
\begin{equation}
\label{equation3.17}
t^m \left( 1 + \sum_{\wtj = 1}^{\infty} c_\wtj t^\wtj \right), 
\qquad\qquad c_\wtj \in \ringRl \text{ of weight } \wtj.
\end{equation}
In characteristic zero, an
analogous definition holds for series expressed in terms of the analytic
uniformizer $z$.  By the following lemma, the condition
of being $\ringRl$-balanced does not depend on whether one expands
with respect to $t$ or $z$.
\end{definition}

\begin{lemma}
\label{lemma3.8}
\begin{enumerate}
\item 
If $f(t)$ and $g(t)$ are $\ringRl$-balanced Laurent series, then so are
$f(t)g(t)$ and $f(t)/g(t)$.
\item
If $f(t) = t^m(1 + c_1 t + \cdots)$ is $\ringRl$-balanced, with $n | m$ and
$n$ invertible in $\fieldk$,
then the ``principal branch'' of the $n$th root $f(t)^{1/n} =
t^{m/n}(1 + c_1 t/n + \cdots)$ is again $\ringRl$-balanced.
\item
Assume that $\fieldk$ has characteristic $0$.  Then
$z = z(t) = t + 2\acoeff t^5/5 + \cdots$ 
and $t=t(z) = z - 2\acoeff z^5/5 + \cdots$ 
are both $\ringR_1$-balanced series.  It follows that a series $f(t)$ is
$\ringRl$-balanced if and only if $f(t(z))$ is.
\item
If $f(t) = t^m(1 + c_1 t + \cdots)$ is $\ringRl$-balanced, then the
logarithmic differential $df/f$ has the expansion
$df/f = t^{-1}(m + \sum_{\wtj \geq 1} d_\wtj t^\wtj)dt$ with $d_\wtj$ a
weight~$\wtj$ element of $\ringRl$.
\end{enumerate}
\end{lemma}
\begin{proof}
The first two assertions are elementary.  The third follows because the
invariant differential $\omega = dx/(2y) = dz$ has, by the first assertion,
an $\ringR_1$-balanced expansion $\omega = (1 + 2\acoeff t^4 + \cdots) dt$;
now integrate to obtain that $z=z(t)$ is balanced.  The rest is immediate.
\end{proof}

We can now give the algebraic analog of Definition~\ref{definition2.7} and
Remark~\ref{remark2.11}.
\begin{definition}
\label{definition3.8.2}
Let $\divD$ be a divisor supported on $\El$, with $m_0$ the
multiplicity of $\Pzero$ in $\divD$.  If $\divD$ is principal, we
define $\fD \in \fieldbigK(\E)$ by requiring, analogously
to~\eqref{equation2.16}, that $\Divisor(\fD) = \divD$ and 
$\fD = t^{m_0}(1 + O(t)) \in t^{m_0} (1 + t\Ohat)$.  This is
compatible with our previous normalization when $\fieldk = \C$, since
$t = z + O(z^5)$ by~\eqref{equation3.6}.

When $\divD$ is not principal, define as before $\divD_z = \divD - (\deg
\divD)(\Pzero)$, and consider the principal divisor $\myl \divD_z$.  Then
define
\begin{equation}
\label{equation3.10}
\fD = (\thef_{\myl\divD_z})^{1/\myl}
    = t^{m_0 - \deg \divD}(1 + O(t))
                          \in t^{m_0 - \deg\divD}(1 + t\Ohat),
\end{equation}
using the formal $\myl$th root of the power series.  For $\divD$ principal,
this is the same as the definition a few lines above,
because~\eqref{equation2.16.5} still holds.  Moreover, $\thef_\Pzero =
1$, and $\fD$ is unchanged if we add a multiple of $\Pzero$ to $\divD$.
\end{definition}

It will be convenient to have names for the first few coefficients of the
$t$-expansion of $\fD$.  Let us therefore define $\lambdaD, \muD, \nuD$ in
general by
\begin{equation}
\label{equation3.8}
\fD = t^{m_0 - \deg \divD}
      ( 1 + \lambdaD t + \muD t^2 + \nuD t^3 + \cdots).
\end{equation}
(The expansion in~\eqref{equation3.14} below shows
that this new definition of the symbols $\lambdaD$ and $\nuD$ agrees
with that of
Definition~\ref{definition3.0.1} for the divisors
considered there.)
We also note that~\eqref{equation2.16.5} implies
various relations among the $\lambdaD, \muD, \nuD$, most notably
\begin{equation}
\label{equation3.10.5}
   \thelambda_{\divD+\divE} = \lambdaD + \lambdaE.
\end{equation}
In particular, for $\divD = (\pointP) + (\pointQ) + (\pointR) -
3(\Pzero)$ as in Definition~\ref{definition3.0.1},
\begin{equation}
\label{equation3.10.7}
\lambdaD = \thelambda_{(\pointP) + (\pointQ) + (\pointR)}
         = \lambdaP + \lambdaQ + \lambdaR.
\end{equation}

\begin{theorem}
\label{theorem3.9}
\begin{enumerate}
\item
Let $\divD$ and $\fD$ be as in Definition~\ref{definition3.8.2}.  Then the
algebraic Laurent expansion of $\fD$ in terms of $t$ is 
an $\ringRl$-balanced Laurent series.
\item
The same result holds if we expand $\fD$ with respect to the analytic
uniformizer $z$ in characteristic zero, as well if we expand the
logarithmic derivative $d\fD/\fD$.  Thus if $\fieldk = \C$, this theorem
combined with Theorem~\ref{theorem2.8} and
Proposition~\ref{proposition2.10} imply that all
Eisenstein series on $\Gammal$ belong to $\ringRl$.
\end{enumerate}
\end{theorem}
\begin{proof}
By part (2) of Lemma~\ref{lemma3.8}, with $n=\myl$, we 
reduce to the case where $\divD$ is principal.  Now a principal $\divD$
that is supported on $\El$ can be written as a $\Z$-linear combination of
divisors of two types: (i) $\divD = (\pointP) + (\ominus \pointP) - 2
(\Pzero)$, for $\pointP \in \El - \{\Pzero\}$, and
(ii) $\divD = (\pointP) + (\pointQ) + (\pointR) - 3(\Pzero)$, as in
Definition~\ref{definition3.0.1}.
Thus part (1) of Lemma~\ref{lemma3.8} reduces our task to showing that
$\fD$ is $\ringRl$-balanced for divisors of types (i) and (ii).  The
statements for $d\fD/\fD$ and for $z$-expansions follow similarly.

In case (i), using~\eqref{equation3.7}, we have the $\ringRl$-balanced
expansion
\begin{equation}
\label{equation3.13}
\begin{split}
\thef_{(\pointP) + (\ominus \pointP)}
&= \thef_{(\pointP) + (\ominus \pointP) - 2(\Pzero)}
= x - \xP \\
& = t^{-2}(1 - \xP t^2 - \acoeff t^4 + \cdots)
      \in t^{-2} \ringR_1[\xP][[t]] \subset t^{-2} \ringRl[[t]].
\end{split}
\end{equation}

Similarly, in case (ii) we have the following expansion:
\begin{equation}
\label{equation3.14}
\begin{split}
\thef_{(\pointP) + (\pointQ) + (\pointR)} &= \fD = -y + \lambdaD x + \nuD \\
& =  t^{-3}(1 + \lambdaD t + \nuD t^3 - \acoeff t^4 + \cdots)
  \in t^{-3} \ringR_1[\lambdaD, \nuD][[t]],
\end{split}
\end{equation}
and again $\ringR_1[\lambdaD, \nuD][[t]] \subset \ringRl[[t]]$.
\end{proof}

\begin{remark}
\label{remark3.9.5}
The fact that all Eisenstein series of weights $\geq 2$ belong to
$\ringRl$ can alternatively be proved as in Sections 10.2-10.5
of~\cite{ShimuraElementary}, by 
expressing the higher derivatives of $\wp$ in terms of $\wp$, $\wp'$, and
$\acoeff(\tau)$; this expresses Eisenstein series of weights $4$ and
above in terms of the forms $\xP$, $\yP$, and $\acoeff$.
\end{remark}

\begin{corollary}
\label{corollary3.6}
Let $\fieldk = \C$, and take a divisor $\divD$ supported on $\El$ as
usual.  (The basic case is $\divD = \pointP$.)  Then
\begin{equation}
\label{equation3.15.5}
\lambdaD(\tau) = - \Eone(\tau, \divD).
\end{equation}
In particular, we can also take $\divD = (\pointP) + (\pointQ) +
(\pointR) - 3(\Pzero)$ as in Definition~\ref{definition3.0.1}, and
choose a principal lift 
$\tilde{\divD} = (\alpha) + (\beta) + (\gamma) - 3(0)$, i.e., $\alpha
+ \beta + \gamma = 0$.  We obtain our alternative proof
of~\eqref{equation1.1}, in light of \eqref{equation3.5}
and~\eqref{equation3.5.2}:
\begin{equation}
\label{equation3.15}
\lambdaD(\tau)
  = -\Eone(\tau,\alpha) - \Eone(\tau,\beta) - \Eone(\tau,\gamma) 
  = -\zeta(\alpha) - \zeta(\beta) - \zeta(\gamma).
\end{equation}
\end{corollary}
\begin{proof}
From~\eqref{equation3.8}, we have that
$d\fD/\fD = t^{-1}(m_0 - \deg \divD + \lambdaD t + \cdots)$, which
equals $z^{-1}(m_0 - \deg \divD + \lambdaD z + \cdots)$ because $t$
and $z$ agree up to $O(z^4)$.  We obtain~\eqref{equation3.15.5} 
from~\eqref{equation2.18} (by Remark~\ref{remark2.11}, we may use
nonprincipal $\divD$; note that $\Eone(\tau,\Pzero) = 0$, so
$\Eone(\tau,\divD) = \Eone(\tau,\divD_z)$).  For~\eqref{equation3.15},
use also~\eqref{equation2.13}.
\end{proof}

\begin{remark}
\label{remark3.14}
Theorems \ref{theorem3.12} and~\ref{theorem3.9}
show that when $\myl \geq 3$, all the modular forms
that we have constructed through Laurent expansions can be expressed as
polynomials in the $\{\lambdaD\}$ for $\divD$ as in
Definition~\ref{definition3.0.1}, which are special Eisenstein series
of weight~$1$ when $\fieldk = \C$.  It is equally useful to consider
the $\{\lambdaP \mid \pointP \in \El\}$ as a set of generators of
$\ringRl$.  By~\eqref{equation3.10.7},
the $\{\lambdaD\}$ are linear combinations of the $\{\lambdaP\}$.
Our theorems prove in a rather roundabout way that the $\{\lambdaP\}$
are expressible in terms of the $\{\lambdaD\}$.  On can also see this
directly, by observing that $\myl$ is invertible in $\fieldk$ and that 
$\myl \lambdaP
   = \sum_{n = 1}^{\myl-2}
       \thelambda_{(\pointP) + ([n]\pointP) + ([-n-1]\pointP)}$.
Alternatively, one can express $\lambdaP$ as a linear combination of
$O(\log \myl)$ different $\lambdaD$s, using values of $n$ starting
from $1$ and increasing by a ``double-and-add'' approach until we
reach $n = \myl - 1$.
\end{remark}

We conclude this section by noting various elementary algebraic
relations between the modular forms in $\ringRl$.  We have already
noted that $\thef_\Pzero = 1$; hence
\begin{equation}
\label{equation3.30}
\thelambda_\Pzero = \themu_\Pzero = \thenu_\Pzero = 0.
\end{equation}
Even though $\Pzero$ does not have affine coordinates, it is
convenient to define also
\begin{equation}
\label{equation3.31}
x_\Pzero = y_\Pzero = 0.
\end{equation}
Then, for all $\pointP \in \El$, we have
\begin{equation}
\label{equation3.12.3}
  \thelambda_{\ominus \pointP} = -\lambdaP,
\quad
  \themu_{\ominus\pointP} = \muP,
\quad
  x_{\ominus\pointP} = \xP,
\quad
  \thenu_{\ominus\pointP} = -\nuP,
\quad
  y_{\ominus\pointP} = -\yP.
\end{equation}
This comes from considering the automorphism of~\eqref{equation3.12}
for $u=-1$, which sends $\pointP$ to $\ominus \pointP$, and acts like
$(-1)^\wtj$ on a modular form of weight~$\wtj$.
With our conventions, we also have the identities
\begin{equation}
\label{equation3.12.6}
\sum_{\pointP\in\El} \lambdaP
 =  \sum_{\pointP\in\El} \muP
 =  \sum_{\pointP\in\El} \xP
 =  \sum_{\pointP\in\El} \nuP
 =  \sum_{\pointP\in\El} \yP
 = 0.
\end{equation}
For odd weights ($\lambdaP, \nuP, \yP$), this is clear.  Morally
speaking, \eqref{equation3.12.6}~holds because each sum above is a
modular form of full level $\Gamma(1)$ of weight $1$, $2$, or~$3$, and
is hence zero.  This can be turned into an algebraic proof, by
considering 
the Galois group
$\Gal(\fieldbigKl/\fieldbigK) \isomorphic SL(2,\Z/\myl\Z)$ and its
natural action on $\El$.  The sums thus belong to $\fieldbigK$, and
one can show that they are integral over the unique factorization
domain $\ringR_1$, hence belong to $\ringR_1$, hence are zero because
of their low weight.

We can also give the following direct proof that $\sum \xP = 0$: the
sum is essentially the coefficient of  $x^{\myl^2 - 2}$ in the
polynomial $f(x) = \prod_{\pointP \in \El-\{\Pzero\}} (x - \xP)$.
But $f(x)$ is a constant multiple of the
square\footnote{
We always have $\psi^2_\myl \in \ringR_1[x]$, since $\psi_\myl \in
\ringR_1[x]$ or $y\ringR_1[x]$ for $\myl$ odd or even, respectively.
}
of the $\myl$-division polynomial:
$\psi_\myl(x,y)^2 = \myl^2 f(x)
= \myl^2 x^{\myl^2 - 1} + \cdots \in \ringR_1[x]$
(see, for example, Exercise~III.3.7 of~\cite{SilvermanI}).
Now by an analog of $\ringR_1$-balanced series for polynomials, the
coefficient of $x^{\myl^2 - 2}$ in $\psi_\myl(x,y)^2$ is known to be
a weight~$2$ element of $\ringR_1$, and is hence zero.
Finally, we defer the proof that $\sum \muP = 0$ to the proof of
Proposition~\ref{proposition4.3}.

We collect the last few identities of this section in a lemma.
Note that~\eqref{equation3.22} below has already appeared for
$\Gamma_1(\myl)$ 
in~\cite{BorisovGunnellsNonvanishing,BorisovGunnellsPopescuEquations}.
The approach of obtaining relations by taking a sum of residues over
all points of $\E$ is taken from~\cite{BorisovGunnells}.
\begin{lemma}
\label{lemma3.15}
\begin{enumerate}
\item
Let $\pointP \in \El - \{\Pzero\}$.  Then the Laurent expansion of the
logarithmic differential $d\fP/\fP$ begins with
\begin{equation}
\label{equation3.20.5}
d\fP/\fP = t^{-1}[-1 + \lambdaP t - \xP t^2 + \yP t^3 + \cdots]dt.
\end{equation}
(This is the algebraic analog of~\eqref{equation2.19.5}, taking into
account \eqref{equation3.5}, \eqref{equation3.6},
and~\eqref{equation3.15.5}.)
We deduce the following equations, which over $\C$ can also be seen
from~\eqref{equation2.21}:
\begin{equation}
\label{equation3.21}
\xP = \lambdaP^2 - 2\muP, \qquad
\yP = 3\nuP - 3\muP\lambdaP + \lambdaP^3.
\end{equation}
\item
Let $\divD = (\pointP) + (\pointQ) + (\pointR)$ be as usual a divisor
supported on $\El - \{\Pzero\}$ with $\oplus \divD = \Pzero$.  Then
\begin{equation}
\label{equation3.22}
\lambdaP\lambdaQ + \lambdaQ\lambdaR + \lambdaP\lambdaR
 + \muP + \muQ + \muR = 0.
\end{equation}
\end{enumerate}
\end{lemma}
\begin{proof}
For \eqref{equation3.20.5} and~\eqref{equation3.21}, consider the
meromorphic differential form $d\fP/\fP$ on $\E$.  Recall that
$\fP = (\thef_{\myl(\pointP) - \myl(\Pzero)})^{1/\myl}$ exists in
$\Ohat$ but is not a meromorphic function on $\E$; however, its
logarithmic differential makes sense globally on $\E$.
Now $d\fP/\fP$ has simple poles at each of $\Pzero$
and $\pointP$, with residues $-1$ and $1$, respectively.  
The sum of the residues of the global meromorphic
differential $x\,d\fP/\fP$ (respectively, $y\,d\fP/\fP$) at all points
of $\E(\fieldbigKbar)$ is zero.  Taking into account the fact that $x
= t^{-2}(1+O(t^4))$ and $y = -t^{-3}(1+O(t^4))$, this
yields the coefficients $\xP$ and $\yP$ in~\eqref{equation3.20.5}.
On the other hand, we can directly compute the logarithmic
differential of $\fP = t^{-1}(1+\lambdaP t + \muP t^2 + \nuP t^3 +
\cdots)$, and this yields the coefficient $\lambdaP$
in~\eqref{equation3.20.5}, as well as~\eqref{equation3.21}.
Finally, to see~\eqref{equation3.22}, combine the 
equations $\xP = \lambdaP^2 - 2\muP$ for $\pointP$, $\pointQ$, and
$\pointR$ with~\eqref{equation3.5.4}. 
\end{proof}

\section{Relations involving the Weil pairing and Hecke operators}
\label{section4}

In this section, we prove deeper algebraic relations between the
modular forms in $\ringRl$.
The
first few relations arise from the Weil pairing on the $\myl$-torsion
group $\El$ of our elliptic curve.  Other relations are related
to the action of the full Hecke algebra of $\Gammal$ on modular forms
in $\ringRl$.  We eventually obtain enough relations to be able to
show in essence that the weight $2$ and~$3$ parts of
$\ringRl$ are stable under the action of the Hecke algebra.
(Actually, in the case of weight~$3$ we obtain only a partial result
at this stage of the proof.)
We use this in Section~\ref{section5} to conclude over $\C$ that 
the ring $\ringRl$ contains all modular forms of weights $2$ and
above.  This of course implies Hecke stability in all weights, and
supersedes the previous result.

The overall shape of our formulas related to Hecke
operators is similar to the results in the articles of Borisov and
Gunnells~\cite{BorisovGunnells,BorisovGunnellsNonvanishing,BorisovGunnellsHigherWeight}.
Those articles work with $\Gamma_1(\myl)$, and prove their formulas via
$q$-expansions. 
Our treatment of $\Gammal$ proceeds instead from the modular
parametrization given 
by the modular curve.  We hope to treat some of the connections
between our approach and theirs, as well as the results in~\cite{Pasol},
in later work; it would also be desirable to understand the Hecke 
action better by directly connecting our relations from Laurent
expansions to the geometry of toric varieties used in 
\cite{BorisovGunnells}.

Before introducing the Weil pairing on $\El$, we 
discuss pullbacks (i.e., composition) of elements $\Ohat$ by the
multiplication map $[n]:\E\to\E$, so as to be able to define the
element $\fQ\compose[n]\in \Ohat$.  This can be done
entirely inside the formal group, since we have an expansion of the form
$t \compose [n] = nt + 2at^5(n-n^5)/5 + O(t^7) \in \ringR_1[[t]]$, so we
can obtain the Laurent expansion
$\fQ\compose[n] 
  = n^{-1}t^{-1}(1 + \lambdaQ nt + \cdots)$.
Another approach is to realize $\fQ\compose[n]$ as the formal $\myl$th root
of (a constant times) the global meromorphic function
$\thef_{\myl(\pointQ)-\myl(\Pzero)}\compose[n]$; this last function is
determined by its divisor, which is $\myl$ times the divisor $\divD$
of~\eqref{equation4.0.5} below.

\begin{definition}
\label{definition4.1}
\begin{enumerate}
\item 
Let $\pointQ \in \El - \{\Pzero\}$ and let $1 \leq n \in \Z$, with $n$
invertible in $\fieldk$.  Choose
a point $\pointQ' \in \E[n\myl]$ such that $[n]\pointQ' = \pointQ$.
Then define the element $\fQ \compose [n] := n^{-1}\fD \in \Ohat$, where
\begin{equation}
\label{equation4.0.5}
\divD 
   = \sum_{\pointT\in\E[n]} (\pointQ' \oplus \pointT)
       - \sum_{\pointT \in \E[n]}(\pointT)
   = [n]^*\bigl( (\pointQ) - (\Pzero) \bigr).
\end{equation}
We have the Laurent expansion
\begin{equation}
\label{equation4.1}
\fQ \compose [n]
   = n^{-1} t^{-1} (1 + \lambdaQ n t + \muQ n^2 t^2
       + \nuQ n^3 t^3 + O(t^4)) \in t^{-1} \ringRl[[t]].
\end{equation}
(Caution: the terms absorbed into $O(t^4)$ do \emph{not} follow the simple
initial pattern.)
Incidentally, $\thef_{\Pzero} = \thef_{\Pzero}\compose[n] = 1$.
\item
In the special case $n = \myl$, define
$\gQ = \fQ \compose[\myl]$.  The divisor $\divD$ of~\eqref{equation4.0.5}
is now principal, so $\gQ \in \fieldbigKl(\E)$ is a global meromorphic
function.
\item
The Weil pairing $e_\myl:\El\times\El \to \boldmu_\myl$ is given (as
usual) by the behavior of the functions $\gQ$ under translation by
elements of $\El$: namely,
\begin{equation}
\label{equation4.2}
\gQ(\pointP \oplus \pointR) = e_\myl(\pointQ,\pointR)
\gQ(\pointP), \text{ where } \pointQ, \pointR \in \El \text{ and }
\pointP \in  \E(\fieldbigKbar).
\end{equation}
\end{enumerate}
\end{definition}
\begin{remark}
\label{remark4.2}
If $\fieldk=\C$, consider the case when
$\pointQ = \pointP_{1/\myl}$ and $\pointR = \pointP_{\tau/\myl}$.
One can then show that our normalization gives
$e_\myl(\pointP_{1/\myl},\pointP_{\tau/\myl}) = e^{2\pi i/\myl}$.
(The easiest way to do this calculation is to avoid the Weierstrass
$\sigma$-function; instead, begin by showing that
$\theg_{\pointP_{1/\myl}}(z) = C \cdot \vartheta(\myl
z-1/\myl)/\vartheta(\myl z)$ for some nonzero constant $C$, where
$\vartheta = \vartheta_{11}$.)
\end{remark}

In weight~$1$, the Weil pairing gives rise to a subtle symmetry 
between the $\{\lambdaP\}$, essentially a duality under the Fourier
transform on $\El$ with respect to $e_\myl$.  When $\fieldk = \C$, this
subtle symmetry motivates Hecke's result that the dimension of the
space of Eisenstein series of weight~$1$ on $\Gammal$ is half the
number of cusps of $X(\myl)$ (see the end of Section~2
of~\cite{HeckeEisenstein}).  This symmetry is usually
expressed in terms of $q$-expansions of weight~$1$ Eisenstein series;
see the second identity at the beginning of Section~7
of~\cite{HeckeZurTheorie}, or the treatment in Sections 3.4
and~3.5 of~\cite{Katz}.

\begin{proposition}
\label{proposition4.3}
The following identities hold for all $\pointR \in \El$:
\begin{equation}
\label{equation4.3}
\begin{split}
\lambdaR &=
 \frac{-1}{\myl} \sum_{\pointQ \in \El}
                           \lambdaQ e_\myl(\pointQ,\pointR),\\
\xR &= - \sum_{\pointQ \in \El}
                           \muQ e_\myl(\pointQ,\pointR),\\
\yR &= -\myl \sum_{\pointQ \in \El}
                           \nuQ e_\myl(\pointQ,\pointR).\\
\end{split}
\end{equation}
(Note that $\sum_{\pointQ} = \sum_{\pointQ \in \El - \{\Pzero\}}$,
by \eqref{equation3.30} and~\eqref{equation3.31}.)
Also, by Fourier inversion,
\begin{equation}
\label{equation4.4}
 \muR = \frac{-1}{\myl^2} \sum_{\pointQ \in \El}
                             \xQ e_\myl(\pointQ,\pointR),
\qquad\qquad
 \nuR = \frac{-1}{\myl^3} \sum_{\pointQ \in \El}
                             \yQ e_\myl(\pointQ,\pointR).
\end{equation}
\end{proposition}
\begin{proof}
Let $\pointQ \in \El - \{\Pzero\}$, and consider 
$\gQ$ as in Definition~\ref{definition4.1}, with its Laurent
expansion as in~\eqref{equation4.1} for $n=\myl$.  Define the global 
meromorphic differential form $\eta_\pointQ = g_\pointQ \omega$ on
$\E$, where $\omega = (1+O(t^4))dt$ is the invariant differential; the
only singularities of $\eta_\pointQ$ are simple poles at the points of
$\El$.  Now the residue of $\eta_\pointQ$ at $\Pzero$ is
$\myl^{-1}$, and~\eqref{equation4.2} says that
$\trstar \eta_\pointQ = e_\myl(\pointQ, \pointR) \eta_\pointQ$, where 
$\tau_\pointR: \E \to \E$ is translation by $\pointR$.  Thus the
residue of $\eta_\pointQ$ at any $\pointR \in \El$ is
$\myl^{-1} e_\myl(\pointQ, \pointR)$.  Now define the differential
form $\eta = -\myl \sum_{\pointQ \in \El - \{\Pzero\}} \eta_\pointQ$.
Nondegeneracy of the Weil pairing implies that $\eta$ has simple poles
at all the points of $\El$, and 
that the residue of $\eta$ at $\Pzero$ is $-\myl^2 + 1$, while the
residue at $\pointR \in \El - \{\Pzero\}$ is $1$.
Moreover, we have the following series expansions of $\eta$ and
$\trstar \eta$ for $\pointR \neq \Pzero$ (the sums are over
$\pointQ \in \El - \{\Pzero\}$):
\begin{equation}
\label{equation4.5}
\begin{split}
\eta &= t^{-1}\Bigl[(-\myl^2 + 1)
                      - \sum_\pointQ \lambdaQ \myl t 
                      - \sum_\pointQ \muQ \myl^2 t^2
                      - \sum_\pointQ \nuQ \myl^3 t^3
                      + \cdots\Bigr] dt,\\
\trstar \eta 
  & = t^{-1}\Bigl[1
              - \sum_\pointQ \lambdaQ e_\myl(\pointQ, \pointR) \myl t
              - \sum_\pointQ \muQ e_\myl(\pointQ, \pointR) \myl^2 t^2
              - \sum_\pointQ \nuQ e_\myl(\pointQ, \pointR) \myl^3 t^3
              + \cdots \Bigr] dt.\\
\end{split}
\end{equation}
The expansion of $\trstar \eta$ holds because
$\trstar \eta =  
  - \myl \sum_\pointQ e_\myl(\pointQ, \pointR) \eta_\pointQ$.
Note that $\eta = t^{-1}((-\myl^2+1) + O(t^2))dt$
(since $\thelambda_{\ominus \pointQ} = -\lambdaQ$, so
$\sum_\pointQ \lambdaQ = 0$.)
We now claim (nontrivially) that
$\eta = d\fD/\fD$, where $\fD$ corresponds to the principal divisor 
$\divD = \Bigl(\sum_{\pointQ \in \El} (\pointQ) \Bigr) - \myl^2 (\Pzero)
   = \Bigl( \sum_{\pointQ \in \El - \{\Pzero\}} (\pointQ) \Bigr)
        + (-\myl^2 + 1) (\Pzero)$.
Indeed, note that
$\eta$ and $d\fD/\fD$ have simple poles at
the same locations, with the same residues.  Therefore $\eta-d\fD/\fD$ is
globally holomorphic, hence constant; let us show that the
difference vanishes at $\Pzero$.  We have
$\fD = \pm \myl^{-1}\psi_\myl(x,y)$ where $\psi_\myl$ is the $\myl$th
division polynomial.  Hence $\fD$ has an $\ringR_1$-balanced
Laurent expansion at $\Pzero$ of the form $\fD = t^{-\myl^2 + 1}(1 +
O(t^4))$, because $\ringR_1$ does not contain elements of degree less than
$4$. 
This shows that $d\fD/\fD = t^{-1}[(-\myl^2 + 1) + O(t^4)]dt$, and
proves our claim.  We obtain that $\eta =
t^{-1}((-\myl^2+1) + O(t^4))dt$, thereby completing the proof
of~\eqref{equation3.12.6}.  This also proves~\eqref{equation4.3} for
$\pointR=\Pzero$.

Now let $\pointR \in \El - \{\Pzero\}$, and consider the translation of the
equality $\eta = d\fD/\fD$ by $\pointR$.
This gives us $\trstar \eta = d(\trstar \fD)/\trstar \fD$.
The expansion of $\trstar \eta$ is given by~\eqref{equation4.5}.
The expansion of $d(\trstar \fD)/\trstar \fD$ can be computed
from the zeros and poles of $\fD$.  Indeed, we have
$\trstar \fD = C \cdot \fD \cdot (\thef_{(\ominus \pointR)})^{-\myl^2}$
for some nonzero constant $C$.  (Here $\thef_{(\ominus \pointR)}$ is not
a global meromorphic function on $\E$, but
$(\thef_{\ominus \pointR})^{-\myl^2}$ is fine.)
Hence $d(\trstar \fD)/\trstar \fD =
d\fD/\fD - \myl^2 d\thef_{\ominus\pointR}/\thef_{\ominus\pointR}$.
However, from \eqref{equation3.20.5} and~\eqref{equation3.12.3}, we have
\begin{equation}
\label{equation4.6}
d\thef_{{\ominus\pointR}}/\thef_{{\ominus\pointR}}
 = t^{-1}[-1 - \lambdaR t - \xR t^2 - \yR t^3 + \cdots]dt.
\end{equation}
Combining all this and comparing the Laurent expansions
in $\trstar \eta = d(\trstar \fD)/\trstar \fD$, we
obtain~\eqref{equation4.3} as desired.  Equation~\eqref{equation4.4}
then follows immediately.
\end{proof}

The relations~\eqref{equation4.4}, when combined with
\eqref{equation3.5}, imply that the $\{\muP,\nuP\}$ 
are Eisenstein series of weights $2$ and~$3$, when $\fieldk = \C$.  We
formalize this algebraically.

\begin{definition}
\label{definition4.4}
For $\wtj \in \{1,2,3\}$, we define the algebraic space $\Espace_\wtj$
of Eisenstein series of weight~$\wtj$ by
\begin{equation}
\label{equation4.7}
\Espace_1 = \linalgspan\{\lambdaP \mid \pointP \in \El\},
\qquad
\Espace_2 = \linalgspan\{\xP\},
\qquad
\Espace_3 = \linalgspan\{\yP\}.
\end{equation}
(If we wish to draw attention to the level $\myl$, we will write
$\Espace_\wtj^{\myl}$.)

We deduce from \eqref{equation4.4} and~\eqref{equation3.21} that for
all $\pointP \in \El$,
\begin{equation}
\label{equation4.8}
\muP, \lambdaP^2 \in \Espace_2,
\qquad
\nuP \in \Espace_3.
\end{equation}
From~\eqref{equation3.22}, we also obtain that for $\pointP, \pointQ,
\pointR \in \El$ with $\pointP \oplus \pointQ \oplus \pointR =
\Pzero$,
\begin{equation}
\label{equation4.9}
 \lambdaP \lambdaQ + \lambdaQ \lambdaR + \lambdaP \lambdaR \in \Espace_2.
\end{equation}
Note that in the above equation, the points $\pointP, \pointQ, \pointR$
are allowed to take the value $\Pzero$; for example, if $\pointQ =
\Pzero$, then $\lambdaR = - \lambdaP$, in which
case~\eqref{equation4.9} becomes the statement 
$-\lambdaP^2 \in \Espace_2$ that we know from~\eqref{equation4.8}.
(The result that $\muP$ and $\lambdaP^2$ are Eisenstein series, as
well as the result~\eqref{equation4.9}, were already observed for
$\Gamma_1(\myl)$ in~\cite{BorisovGunnellsNonvanishing}).
\end{definition}

In our treatment of Hecke operators, we shall need the following
identities, which are related to the fact that the trace from
$\Gamma(n\myl)$ to $\Gamma(\myl)$ of an Eisenstein series on
$\Gamma(n\myl)$ is again an Eisenstein series.
\begin{lemma}
\label{lemma4.5}
Let $n \geq 1$ be invertible in $\fieldk$.  Let $\pointP \in
\Enl$ (typically, $\pointP \in \El$), and let $\pointT \in \En$.
Consider the modular forms  
$\thelambda_{\pointP \oplus \pointT}$, $x_{\pointP \oplus \pointT}$, and
$y_{\pointP \oplus \pointT}$ on $\Gamma(n\myl)$.  We then have
\begin{equation}
\label{equation4.10}
\sum_{\pointT \in \En} \thelambda_{\pointP \oplus \pointT}
           = n \thelambda_{[n]\pointP},
\quad
\sum_{\pointT \in \En} x_{\pointP \oplus \pointT}
           = n^2 x_{[n]\pointP},
\quad
\sum_{\pointT \in \En} y_{\pointP \oplus \pointT}
           = n^3 y_{[n]\pointP}.
\end{equation}
We also have
\begin{equation}
\label{equation4.11}
\sum_{\pointT \in \En} \themu_{\pointP \oplus \pointT}
           =  \themu_{[n]\pointP},
\qquad
\sum_{\pointT \in \En} \thenu_{\pointP \oplus \pointT}
           = \frac{1}{n} \, \thenu_{[n]\pointP}.
\end{equation}
\end{lemma}
\begin{proof}
Over $\C$, equation~\eqref{equation4.10} is immediate from the
definition of $\Ej$ in \eqref{equation2.1} and~\eqref{equation2.2},
bearing in mind that $\xP$ is a difference between two $\Etwo$s.  Let us
however give a proof in our  
algebraic setting.  When $\pointP = \Pzero$, \eqref{equation4.10} reduces
to~\eqref{equation3.12.6}.  If $\pointP \neq \Pzero$, we begin by noting
the following identity, which follows by comparing zeros and poles, as well
as the leading coefficients of the Laurent expansions:
\begin{equation}
\label{equation4.12}
\thef_{[n]\pointP} \compose [n]
   = n^{-1} 
     \Bigl(\prod_{\pointT \in \En}
               \thef_{\pointP \oplus \pointT}
     \Bigr) / \fD.
\end{equation}
Here $\fD$ corresponds to the principal divisor 
$\divD = \sum_{\pointT \in \En} (\pointT) - n^2 (\Pzero)$.  As in the proof
of Proposition~\ref{proposition4.3}, we have an expansion 
$\fD = t^{-n^2+1}(1+O(t^4))$.  Now taking the logarithmic differential of
both sides of~\eqref{equation4.12} and comparing the first few
coefficients yields~\eqref{equation4.10}, as desired.

As for~\eqref{equation4.11}, we prove it using the Fourier duality of
Proposition~\ref{proposition4.3}.  (This approach also yields a 
different proof of~\eqref{equation4.10}.)  For instance,
use~\eqref{equation4.4} to express each $\themu$ in the first sum
in~\eqref{equation4.11} in terms of an $x$.  This yields
\begin{equation}
\label{equation4.13}  
\sum_{\pointT \in \En} \themu_{\pointP \oplus \pointT}
  = \sum_{\pointT \in \En} \frac{-1}{n^2\myl^2}
                   \sum_{A \in \Enl}
                        x_A e_{n\myl}(A,\pointP \oplus \pointT).
\end{equation}
Rearrange the sum as $\sum_A \sum_\pointT$, and use the property of
the Weil pairing 
\begin{equation}
\label{equation4.14}
  A \in \Enl, \quad \pointT \in \En
  \implies
 e_{n\myl}(A,\pointT) = e_n([\myl]A,\pointT)
\end{equation}
to conclude that the only surviving terms are those when $[\myl]A =
\Pzero$, in other words, for $A \in \El$.  Thus we obtain
\begin{equation}
\label{equation4.15}
\sum_{\pointT \in \En} \themu_{\pointP \oplus \pointT}
  = \frac{-n^2}{n^2\myl^2} \sum_{A \in \El} x_A e_{n\myl}(A,\pointP)
  = \frac{-1}{\myl^2} \sum_{A \in \El} x_A e_{\myl}(A,[n]\pointP),
\end{equation}
where the last equality is analogous to~\eqref{equation4.14}.  This
implies the first part of~\eqref{equation4.11}.  The second part,
involving $\thenu$, is proved similarly.
\end{proof}

The following is the main ingredient in our proof that the
degree~$2$ part of $\ringRl$ is stable under the Hecke algebra.  The
argument involves an interesting induction on the level.  We start with
forms on $\Gamma(n\myl)$, ``raise the level'' to rewrite them in terms of
forms on $\Gamma(sn\myl)$ with $s < n$, ``lower the level'' back 
to $\Gamma(s\myl)$, and repeat.

\begin{proposition}
\label{proposition4.6}
Let $n \geq 1$ and assume that $n!$ is invertible in $\fieldk$.  Let
$A,B \in \Enl$ (as before, typically $A,B \in \El$), and let $s \in
\Z$.  Then
\begin{equation}
\label{equation4.16}
\begin{split}
&\sum_{T \in \En}
   \thelambda_{A \oplus T}
   \thelambda_{B \ominus [s]T} \\
&\quad=
\bigl(\text{a linear combination of terms of the form } 
         \thelambda_{[a]A \oplus [b]B} \thelambda_{[c]A \oplus [d]B}
\bigr)\\
&\qquad+
\bigl(\text{an element of } \Espace^{n!\myl}_2
\bigr),\\
\end{split}
\end{equation}
where the linear combination above is over finitely many
$(a,b,c,d) \in \Z^4$ satisfying
\begin{equation}
\label{equation4.17}
\det \twomatr{a}{b}{c}{d} = \pm n,
\qquad\qquad
a - sb \equiv c - sd  \equiv 0 \pmod{n}.
\end{equation}
\end{proposition}
\begin{proof}
The proof is by induction on $n$, the case $n=1$ (so $T =
\Pzero$) being trivial.  Note that the value of $s$ only matters
modulo $n$, so we henceforth assume that $0 \leq s < n$.  If $s = 0$,
then the sum over $T$ is $n \thelambda_{[n]A} \thelambda_B$
by~\eqref{equation4.10}, so we are done.  If $s>0$, we
reduce~\eqref{equation4.16} for the pair $(n,s)$ to the analogous
statement for $(s,n)$, hence for $(s, n \bmod s)$.  (This resembles 
the Euclidean algorithm.)
To this end, choose a point $B' \in \E[sn\myl]$ for which
$[s]B' = B$.  We then see from~\eqref{equation4.10} that 
\begin{equation}
\label{equation4.18}  
\thelambda_{B \ominus [s]T} 
= s^{-1} \sum_{U \in \E[s]} \thelambda_{B' \ominus T \oplus U}.
\end{equation}
Hence, up to the factor $s^{-1}$, our sum in~\eqref{equation4.16}
becomes
\begin{equation}
\label{equation4.19}
\begin{split}
\sum_{T \in \En, \, U \in \E[s]} &\thelambda_{A \oplus T}
                              \thelambda_{B' \ominus T \oplus U}\\
\equiv
\sum_{T,U} &\thelambda_{A \oplus T} \thelambda_{A \oplus B' \oplus U}
  - \sum_{T,U} \thelambda_{\ominus B' \oplus T \ominus U}
             \thelambda_{A \oplus B' \oplus U}
                      \quad \pmod{\Espace_2^{sn\myl}},\\
\end{split}
\end{equation}
where the congruence is obtained from~\eqref{equation4.9} with
$\pointP = A \oplus T$, $\pointQ = B' \ominus T \oplus U$,
and $\pointR = \ominus A \ominus B' \ominus U$; we have also
used~\eqref{equation3.12.3}.  Now the first sum on the right hand side
of equation~\eqref{equation4.19} is a constant (namely, $ns$) times
$\thelambda_{[n]A} \thelambda_{[s](A \oplus B')}
= \thelambda_{[n]A} \thelambda_{[s]A \oplus B}$, which has the desired 
form.  On the other hand, the second sum on the right hand side can be
summed first over all $T \in \En$, which by~\eqref{equation4.10}
yields a constant times 
\begin{equation}
\label{equation4.20}
\sum_{U \in \E[s]}
 \thelambda_{[-n]B' \ominus [n]U} \thelambda_{A \oplus B' \oplus U}.
\end{equation}
By the inductive hypothesis, the above sum is congruent modulo
$\Espace_2^{s! n\myl}$ to a linear combination of terms of the form
\begin{equation}
\label{equation4.21}
 \thelambda_{[a'](A \oplus B') \oplus [-nb'] B'}
 \thelambda_{[c'](A \oplus B') \oplus [-nd'] B'}
=
 \thelambda_{[a']A \oplus [\frac{a'-nb'}{s}]B}
 \thelambda_{[c']A \oplus [\frac{c'-nd'}{s}]B}
\end{equation}
where $(a',b',c',d')$ satisfy~\eqref{equation4.17} with the roles of
$s$ and $n$ interchanged; in particular, $\frac{a'-nb'}{s},
\frac{c'-nd'}{s} \in \Z$, and we get that each term is of the form
$\thelambda_{[a]A \oplus [b]B} \thelambda_{[c]A \oplus [d]B}$,
satisfying the original requirements of~\eqref{equation4.17}.
Finally, we remark that $\Espace_2^{sn\myl}$ and
$\Espace_2^{s!n\myl}$ are both subspaces of $\Espace_2^{n!\myl}$.
\end{proof}
\begin{remark}
\label{remark4.7}
The element of $\Espace_2^{n!\myl}$ above actually belongs to
$\Espace_2^{n\myl}$, but we shall not prove this in our algebraic
context; it is obvious over $\C$, since it is an Eisenstein series
with level $n!\myl$ that happens to transform under $\Gamma(n\myl)$.
(Similarly, if $A,B \in \El$, then the element of $\Espace_2$ above actually
belongs to $\Espace_2^\myl$.) It is possible to specify this element
more precisely by applying~\eqref{equation3.22} (provided
$\pointP,\pointQ,\pointR \neq \Pzero$) and~\eqref{equation4.11} in the
above proof.  This typically yields an
element of $\Espace_2$ that is a linear combination of terms
$\themu_{[a]A+[b]B}$ where $a-sb \equiv 0 \pmod{n}$.
On another topic, we observe that the linear combination
in~\eqref{equation4.16} is $\Z$-linear, with all coefficients
divisible by $n$.
\end{remark}

From now on, we shall for convenience work exclusively over $\C$.  Also, 
since $\ringR_1$ and $\ringR_2$ are the full rings of modular forms on
$\Gamma(1)$ and $\Gamma(2)$, we can restrict to $\myl \geq 3$.  As usual,
for a weight~$\wtj$ and a congruence subgroup $\Gamma$, we write:
\begin{equation}
\label{equation4.22}
  \cuspforms_\wtj(\Gamma) = \{\text{cusp forms}\}
\subset
  \modforms_\wtj(\Gamma) = \{\text{holomorphic modular forms over }\C \}.
\end{equation}
Also, an element $\gamma \in \Gamma(1)$ acts as usual on 
$\modforms_\wtj(\Gammal)$ by $f \mapsto f |_{\wtj} \gamma$,
and preserves both $\cuspforms_\wtj(\Gammal)$ and the Eisenstein
subspace of $\modforms_\wtj(\Gammal)$;
we can equivalently view $\gamma$ as an element of $\Gamma(1)/\Gammal
\isomorphic SL(2,\Z/\myl\Z)$.
Such a $\gamma$ also acts by automorphisms on $\El$ (preserving the
Weil pairing) and the ring $\ringRl$.
We have $\pointP \mapsto \pointP \cdot \gamma$, where
\begin{equation}
\label{equation4.23}
\begin{split}
\zP = \frac{\torsi \tau + \torsj}{\myl}
  &\implies
z_{\pointP \cdot \gamma} = \frac{\torsi' \tau + \torsj'}{\myl}
\text{ with }
(\torsi' \quad \torsj') = (\torsi \quad \torsj) \gamma,
\\
\pointP \in \El 
  &\implies
\lambdaP |_1 \gamma = \thelambda_{\pointP \cdot \gamma}.
\\
\end{split}
\end{equation}

We briefly review the well-known interpretation of Hecke operators in
terms of a trace between congruence subgroups.  Given a Hecke operator
described as a double coset $\Gammal \alpha \Gammal$ with
$\alpha \in GL^+(2,\Q)$, we can harmlessly multiply $\alpha$ by a
scalar to obtain a primitive integral matrix; then composing this
double coset on the left and right by the action of elements
$\gamma_1, \gamma_2 \in \Gamma(1)$ allows us to assume without loss of
generality that $\alpha = \stwomatr{n}{}{}{1}$ for some $n \geq 1$.
We then have, for $f(\tau) \in \modforms_\wtj(\Gammal)$:
\begin{equation}
\label{equation4.24}
f |_\wtj \Gammal \twomatr{n}{}{}{1} \Gammal
 = C \sum_{\gamma \in \Gamma(n\myl)\backslash\Gammal}
                \bigl(f(n\tau) \bigr) |_\wtj \gamma,
\end{equation}
where $C = C_{n,\myl,\wtj}$ is a suitable normalizing constant.  Note
that if $f(\tau) \in \ringRl$, then $f(n\tau) \in \ringR_{n\myl}$;
indeed, the map $f \mapsto f(n\tau)$ respects multiplication of forms,
so it is enough to check the above statement for the weight~$1$
Eisenstein series $\lambdaP = -\Eone(\tau,\pointP)$ that generate
$\ringRl$.  This is just the identity
\begin{equation}
\label{equation4.25}
\Eone(n\tau, \frac{\torsi \tau + \torsj}{\myl})
  = n^{-1}
    \sum_{k \bmod n}
        \Eone(\tau, \frac{\torsi n\tau + \torsj + k\myl}{n\myl}). 
\end{equation}
The sum over representatives $\gamma \in \Gamma(n\myl)\backslash\Gammal$
in~\eqref{equation4.24} is a trace from
$\modforms_\wtj(\Gamma(n\myl))$ to $\modforms_\wtj(\Gamma(\myl))$, and
we shall henceforth work with it instead of with double cosets.

We can now state and prove our result related to Hecke stability of
$\ringRl$ in weight~$2$, which will be superseded later when we show that
$\ringRl \supset \modforms_2(\Gammal)$.

\begin{proposition}
\label{proposition4.8}
Let $\fieldk = \C$.  Then the trace of a weight~$2$ element of
$\ringR_{n\myl}$ from $\modforms_2(\Gamma(n\myl))$ to
$\modforms_2(\Gamma(\myl))$ actually belongs to $\ringRl$.
(\emph{A priori,} this trace merely belongs to
$\ringR_{n\myl} \intersect \modforms_2(\Gammal)$.)
\end{proposition}
\begin{corollary}
\label{corollary4.9}
Over $\C$, the weight~$2$ part of $\ringRl$ is stable under the action
of the Hecke algebra for $\Gammal$.
\end{corollary}
\begin{proof}[Proof of Proposition~\ref{proposition4.8}]
As mentioned above, we can assume that $\myl \geq 3$.  It is then enough to
show that the trace of any product $\lambdaP \lambdaQ = \Eone(\tau,\pointP)
\Eone(\tau,\pointQ)$ with $\pointP, \pointQ \in \Enl - \{\Pzero\}$
belongs to $\ringRl$.  
The ring $\ringRl$ contains all the Eisenstein series on $\Gammal$, so we
can work modulo Eisenstein series throughout.  As observed in
Remark~\ref{remark4.7}, this can be done even if we encounter
Eisenstein series of higher level in some intermediate steps.
Now the trace down from level $n\myl$ to level $\myl$ can be
done one prime factor at a time, so we may assume that $n$
is a prime number.  There are hence two cases to consider: (i) $n$ is prime 
and $n \notdividing \myl$, and (ii) $n$ is prime and $n | \myl$.

In case (i), we decompose $\Enl = \El \bigoplus
\En$, and note that $\Gamma(n\myl)\backslash\Gamma(\myl) \isomorphic
SL(2,\Z/n\Z)$; the action of this group affects only the $\En$ part.
Let 
$\pointP = A \oplus T_0$ and $\pointQ = B \oplus U_0$, with $A,B \in
\El$ and $T_0, U_0 \in \En$.  If $T_0 = U_0 = \Pzero$, then 
$\lambdaP \lambdaQ = \thelambda_A \thelambda_B \in \ringRl$ already.
Otherwise, we have (say) $T_0 \neq \Pzero$.  There are two
subcases: (i.a) there exists $s \in \Z/n\Z$ (possibly $s=0$) such that
$U_0 = [s]T_0$, and (i.b) $\{T_0, U_0\}$ are a basis for $\E[n]$.
In subcase (i.a), the trace of
$\thelambda_{A \oplus T_0} \thelambda_{B \oplus [s] T_0}$ is equal to a
multiple of
$\sum_{T \in \En - \{\Pzero\}}
    \thelambda_{A \oplus T} \thelambda_{B \oplus [s]T}
= \sum_{\text{all }T \in \En} - \thelambda_A \thelambda_B$.
The sum over all $T$ is in $\ringRl$ by
Proposition~\ref{proposition4.6}, and $\thelambda_A \thelambda_B \in
\ringRl$, so we are done.

In subcase (i.b), let
$\zeta = e_n(T_0, U_0)$, a primitive $n$th root of~$1$.  The trace is then
\begin{equation}
\label{equation4.26}
 \sum_{\substack{T,U \in \En \\ e_n(T,U) = \zeta}}
    \thelambda_{A \oplus T} \thelambda_{B \oplus U}
=  \frac{-1}{n\myl}
   \sum_{\substack{T,U \in \En \\ e_n(T,U) = \zeta\\ V \in \En, C \in \El}}
    \thelambda_{A \oplus T} \thelambda_{C \oplus V} 
      e_{n\myl}(C \oplus V, B \oplus U).    
\end{equation}
We have invoked~\eqref{equation4.3} above; note that $C \oplus V$
ranges over the elements of $\Enl$.  
Now $e_{n\myl}(C \oplus V, B \oplus U) = e_{\myl}([n]C,B) e_{n}([\myl]V,U)$,
so~\eqref{equation4.26} equals a linear combination of terms
(indexed by $C$) of the form
\begin{equation}
\label{equation4.27}
  \sum_{\substack{T,U,V \in \En \\ e_n(T,U) = \zeta}}
     \thelambda_{A \oplus T} \thelambda_{C \oplus V}
     e_{n}([\myl]V,U).
\end{equation}
For fixed $T$ and $V$, we must hence study the sum over those $U$ for
which $e_n(T,U) = \zeta$.  Such a $U$ exists if and only if $T \neq \Pzero$
(recall that $n$ is prime and $\zeta \neq 1$), in which case
$U$ ranges over the set
$\{U_T \oplus [t]T \mid t \in \Z/n\Z\}$ for some choice of $U_T$
(depending on $T$) with $e_n(T,U_T) = \zeta$.  The sum over $U$ thus
contains a factor $\sum_{t \in \Z/n\Z} e_n([\myl]V,U_T \oplus [t]T)$, which
vanishes unless $V$ belongs to the cyclic subgroup generated by $T$ (recall
that $n \notdividing \myl$).  We obtain
that~\eqref{equation4.27} is equal to
\begin{equation}
\label{equation4.28}
\begin{split}
  &\sum_{T \in \En - \{\Pzero\}}\quad
  \sum_{V \text{ of the form } V = [s]T}
          \thelambda_{A \oplus T} \thelambda_{C \oplus [s]T} 
           \cdot n e_n([\myl][s]T,U_T)\\
 &\quad= \sum_{s \in \Z/n\Z}
           n \zeta^{\myl s}
         \sum_{T \neq \Pzero} 
            \thelambda_{A \oplus T} \thelambda_{C \oplus [s]T},\\
\end{split}
\end{equation}
which brings us back to subcase (i.a).

We now turn to case (ii), so $\myl = Ln^k$ with
$n \notdividing L$ and $k \geq 1$.  Write
$P = A \oplus T_0$ and $Q = B \oplus U_0$ with $A,B \in \E[L]$ 
and $T_0,U_0 \in \E[n^k]$.  Our trace is a sum over
representatives for $\Gamma(Ln^{k+1})\backslash\Gamma(Ln^k)$.  Such
representatives again do not affect $A$ or $B$, and their action on $T_0$
and $U_0$ can be described by matrices in $SL(2,\Z/n^{k+1}\Z)$ that are
congruent to the identity modulo $n^k$; thus such matrices have the form 
\begin{equation}
\label{equation4.29}
\twomatr{1+n^k \alpha}{n^k \beta}{n^k\gamma}{1-n^k\alpha} 
= I + n^k M, 
\qquad
M = \twomatr{\alpha}{\beta}{\gamma}{-\alpha}
               \in M_2^{\text{trace } 0}(\Z/n\Z).
\end{equation}
Note that although we view the entries of $M$ as being in
$\Z/n\Z$, multiplying them by $n^k$ yields elements of
$n^k\Z/n^{k+1}\Z$, not zero.  We shall feel free to
use other bases for $\E[n^{k+1}] \isomorphic (\Z/n^{k+1}\Z)^2$ than the
standard basis $\{ \pointP_{\tau/n^{k+1}}, \pointP_{1/n^{k+1}} \}$; even if
the change of basis does not have determinant~$1$ (and hence changes the
Weil pairing), our description of $M$ in~\eqref{equation4.29} remains
valid.  Let us write $\hat{T_0} = [n^k]T_0$ and $\hat{U_0} = [n^k]U_0$.  We
have $\hat{T_0},\hat{U_0} \in \E[n]$, and the trace that we want
is then
\begin{equation}
\label{equation4.30}
\sum_{M \in M_2^{\text{trace } 0}(\Z/n\Z)}
  \thelambda_{A \oplus T_0 \oplus (\hat{T_0} \cdot M)}
  \thelambda_{B \oplus U_0 \oplus (\hat{U_0} \cdot M)},
\end{equation}
where the action of $M$ is analogous to that in~\eqref{equation4.23}.  Once
again, the case $\hat{T_0} = \hat{U_0} = \Pzero$ is easy
(since then $T_0, U_0 \in \E[n^k]$ and we are already in 
$\ringR_{n^k L} = \ringRl$), so  without loss of generality $\hat{T_0}
\neq \Pzero$.  We 
face analogous subcases: (ii.a) there exists $s \in \Z/n\Z$ such that
$\hat{U_0} = [s] \hat{T_0}$ and (ii.b) $\{\hat{T_0}, \hat{U_0}\}$ are a
basis for $\E[n]$.

In subcase (ii.a), the points $\hat{T_0}\cdot M$ cover
all of $\E[n]$ (including $\Pzero$), each point $\hat{T} \in \E[n]$
occurring $n$ times.  (The easiest way to see this is to write $M$ with
respect to a basis for $\E[n]$ that includes $\hat{T_0}$.)  Hence we obtain
that~\eqref{equation4.30} is a multiple of
\begin{equation}
\label{equation4.31}
 \sum_{\hat{T} \in \E[n]}
   \thelambda_{A \oplus T_0 \oplus \hat{T}}
   \thelambda_{B \oplus U_0 \oplus [s]\hat{T}}.
\end{equation}
Modulo Eisenstein series, this equals a linear combination of
terms of the form
\begin{equation}
\label{equation4.32}
 \thelambda_{[a](A \oplus T_0) \oplus [b](B \oplus U_0)}
 \thelambda_{[c](A \oplus T_0) \oplus [d](B \oplus U_0)},
\text{ for }
a + sb \equiv c + sd  \equiv 0 \pmod{n}.
\end{equation}
We observe that $[n^k]([a]T_0 \oplus [b] U_0)
 = [a]\hat{T_0} \oplus [b]\hat{U_0} = [a + sb] \hat{T_0} = \Pzero$, whereas
$[a]A \oplus [b]B \in \E[L]$, so the first factor in~\eqref{equation4.32}
is of the form $\thelambda_C$ with $C \in \E[n^k L] = \El$; the second
factor is similar, and we obtain an element of $\ringRl$, as desired.

In subcase (ii.b), we write $M$ in terms of the basis
$\{\hat{T_0}, \hat{U_0}\}$.  Our trace is now
\begin{equation}
\label{equation4.33}
 \sum_{\alpha,\beta,\gamma\in\Z/n\Z}
\thelambda_{A \oplus T_0
                 \oplus [\alpha] \hat{T_0} \oplus [\beta] \hat{U_0}}
\thelambda_{B \oplus U_0
                 \oplus [\gamma] \hat{T_0} \oplus [-\alpha] \hat{U_0}}.
\end{equation}
Similarly to subcase (i.b), we rewrite the factor
$\thelambda_{B \oplus U_0
                 \oplus [\gamma] \hat{T_0} \oplus [-\alpha] \hat{U_0}}$
in terms of the Weil pairing and $\thelambda_{C \oplus V}$, 
for $C \in \E[L]$ and $V \in \E[n^{k+1}]$.  We obtain a linear
combination of terms of the following form (here the triples
$(\alpha,\beta,\gamma) \in (\Z/n\Z)^3$ are analogous to the pairs 
$\{(T,U) \in \E[n]\times\E[n] \mid e_n(T,U) = \zeta\}$
of~\eqref{equation4.27}):
\begin{equation}
\label{equation4.34}
\sum_{\substack{\alpha,\beta,\gamma\in\Z/n\Z \\ V \in \E[n^{k+1}]}}
  \thelambda_{A \oplus T_0
                 \oplus [\alpha] \hat{T_0} \oplus [\beta] \hat{U_0}}
  \thelambda_{C \oplus V}
  e_{n^{k+1}}(
       [L]V, 
       U_0 \oplus [\gamma] \hat{T_0} \oplus [-\alpha] \hat{U_0}
             ).
\end{equation}
Now perform the sum over $\gamma$ first: the inner factor
$\sum_\gamma e_{n^{k+1}}([L]V, [\gamma] \hat{T_0})$
can be rewritten as
$\sum_\gamma e_{n}\Bigl([L]([n^k]V), [\gamma] \hat{T_0}\Bigr)$
with $[n^k]V \in \E[n]$.  Thus, as in case (i.b), the only terms that
survive are those where $[n^k]V = [s] \hat{T_0} = [sn^k] T_0$ for some
$s \in \Z/n\Z$.  Equivalently, we can write $V = [s] T_0 \oplus W$ for
some $s$ and for some $W \in \E[n^k]$.  In such a situation, we have
$e_{n^{k+1}}([L]V, [-\alpha] \hat{U_0})
 = e_n([n^k L]V, [-\alpha]\hat{U_0})
 = e_n([Ls] \hat{T_0},[-\alpha]\hat{U_0})
 = e_n\Bigl([-Ls]([\alpha]\hat{T_0} \oplus [\beta]\hat{U_0}),
            \hat{U_0}\Bigr)$.
At this point, we note that as $\alpha$ and $\beta$ vary, the point
$\hat{T} := [\alpha]\hat{T_0} \oplus [\beta]\hat{U_0}$ runs over all
points of $\E[n]$.
Hence our expression is
a linear combination of terms (indexed by $C$ and $s$) of the form
\begin{equation}
\label{equation4.35}
\sum_{\substack{\hat{T} \in \E[n] \\ W \in \E[n^k]}}
  \thelambda_{A \oplus T_0
                 \oplus \hat{T}}
  \thelambda_{C \oplus [s] T_0 \oplus W}
  e_{n^{k+1}}([Ls]T_0 \oplus [L]W, U_0)
  e_n([-Ls]\hat{T}, \hat{U_0}).
\end{equation}
Each such term contains a common factor
$e_{n^{k+1}}([Ls]T_0,U_0)$.  Also,
\begin{equation}
\label{equation4.36}
e_{n^{k+1}}([L]W, U_0) e_n([-Ls]\hat{T}, \hat{U_0})
= e_{n^{k+1}}\Bigl([L](W \ominus [s]\hat{T}), U_0\Bigr).
\end{equation}
We define $X := W \ominus [s]\hat{T} \in \E[n^k]$; this yields
a bijection from the set $\E[n]\times \E[n^k]$ to itself,
sending the pair $(\hat{T},W)$ to the pair $(\hat{T},X)$.
Our term~\eqref{equation4.35} then becomes
\begin{equation}
\label{equation4.37}
e_{n^{k+1}}([Ls]T_0,U_0)
\sum_{\substack{\hat{T} \in \E[n] \\ X \in \E[n^k]}}
  \thelambda_{A \oplus T_0
                 \oplus \hat{T}}
  \thelambda_{C \oplus [s] T_0 \oplus X \oplus [s]\hat{T}}
 e_{n^{k+1}}([L]X, U_0).
\end{equation}
Writing the sum in the order $\sum_X \sum_{\hat{T}}$, we see that the
inner sum over $\hat{T}$ is now exactly analogous
to~\eqref{equation4.31}.  In our setting, $C$ and $[s]T_0 \oplus X$
play the roles of $B$ and $U_0$ from~\eqref{equation4.31}\footnote{
Recall that $U_0$ in~\eqref{equation4.31} had the property that
$[n^k]U_0 = \hat{U_0} = [s]\hat{T_0}$.  The analogous observation in
our setting is that $[n^k]([s]T_0 \oplus X) = [s]\hat{T_0}$.
}, and we
obtain as in subcase (ii.a) that our final
expression is congruent modulo Eisenstein series to an element of
$\ringRl$.  This ends our proof.
\end{proof}

We next wish to prove weight~$3$ analogs of Propositions
\ref{proposition4.6} and~\ref{proposition4.8}, but only for modular
forms of the form $\xP\lambdaQ$, i.e., for products of two Eisenstein
series of weights $2$ and~$1$.  We again work modulo Eisenstein
series, i.e., modulo $\Espace_3$.  In this context, the analog
of~\eqref{equation4.9} is the following statement
for
$\pointP \oplus \pointQ \oplus \pointR = \Pzero$:
\begin{equation}
\label{equation4.38}
 (\xP - \xR)(\lambdaP + \lambdaQ + \lambdaR) \in \Espace_3.
\end{equation}
To see this, first observe by~\eqref{equation3.5.2} that if
$\pointP, \pointQ, \pointR \neq \Pzero$, then
the above expression is equal to $\yP - \yR \in \Espace_3$.
On the other hand, if one of the points is $\Pzero$, then $\lambdaP +
\lambdaQ + \lambdaR = 0$ by our conventions, so the above expression
is zero.

The next lemma is the weight~$3$ analog of the key computational
step~\eqref{equation4.19} occuring in
Proposition~\ref{proposition4.6}.  The same techniques work for the
weight~$2$ identity  
$(\lambdaP + \lambdaQ + \lambdaR)^2 = \xP + \xQ + \xR \in \Espace_2$; 
they can be used to prove a slightly weaker result than
Proposition~\ref{proposition4.6}, which nonetheless suffices to imply
Proposition~\ref{proposition4.8}.

\begin{lemma}
\label{lemma4.10}
Let $\fieldk = \C$, let $n \geq 1$, and let $A,B \in \E[n\myl]$ (typically
with $A,B \in \El$).  Then we have the following congruences modulo
$\Espace_3$:
\begin{equation}
\label{equation4.39}
 \sum_{T \in \E[n]} x_{A\oplus T} \thelambda_{A \oplus T}
        \equiv
    n x_{[n]A} \thelambda_{[n]A},
\end{equation}
\begin{equation}
\label{equation4.40}
\begin{split}
 \sum_{T \in \E[n]} x_{A \oplus T} \thelambda_{B \ominus T}
        &\equiv
    - nx_{[n]A}\thelambda_{[n]A}
    + n^2 x_{[n]A}\thelambda_{A \oplus B}\\
        & \qquad
    + n x_{A \oplus B} \thelambda_{[n]A}
    + n x_{A \oplus B} \thelambda_{[n]B}
    - n^2 x_{A \oplus B} \thelambda_{A \oplus B}.\\
\end{split}
\end{equation}
\end{lemma}
\begin{proof}
To show~\eqref{equation4.39}, we take
$\pointP = A \oplus T$, $\pointQ = \ominus A \oplus U$, and
$\pointR = \ominus T \ominus U$ in~\eqref{equation4.38}, and we sum the
result over all $T,U \in \E[n]$, knowing that the final result will be
$\equiv 0$ modulo $\Espace_3$.  We now observe that
\begin{equation}
\label{equation4.41}
\begin{split}
\sum_{T,U \in \E[n]} x_{A \oplus T} \thelambda_{A \oplus T}
        &= n^2 \sum_T x_{A \oplus T} \thelambda_{A \oplus T},\\
\sum_{T,U} x_{A \oplus T}\thelambda_{\ominus A \oplus U}
        &= n^2 x_{[n]A} \cdot n \thelambda_{[-n]A}
        = -n^3 x_{[n]A} \thelambda_{[n]A},\\
\sum_{T,U} x_{A \oplus T}\thelambda_{\ominus T \ominus U}
        &= \sum_{T,V \in \E[n]}  x_{A \oplus T}\thelambda_{V} = 0,\\
\sum_{T,U} x_{\ominus T \ominus U} \thelambda_{A \oplus T}
        &= \sum_{T,V} x_V \thelambda_{A \oplus T} = 0;
\text{ similarly, }\sum_{T,U} x_{\ominus T \ominus U}
        \thelambda_{\ominus A \ominus U} = 0,\\
\sum_{T,U} x_{\ominus T \ominus U}  \thelambda_{\ominus T \ominus U}
        &= n^2 \sum_V x_V \thelambda_V
        = n^2 \sum_V x_{\ominus V} \thelambda_{\ominus V} 
        = -(\text{itself}) = 0,
\end{split}
\end{equation}
where we have used~\eqref{equation4.10} and~\eqref{equation3.12.6} as
needed.

For the proof of~\eqref{equation4.40}, we take the sum over
all $T$ in $\E[n]$ of~\eqref{equation4.38} with
$\pointP = A \oplus T$, $\pointQ = B \ominus T$, and
$\pointR = \ominus A \ominus B$
(so $\lambdaR = -\thelambda_{A \oplus B}$ 
and $\xR = x_{A \oplus B}$).  We then proceed as in the
proof of~\eqref{equation4.39}, while using~\eqref{equation4.39} at one
point, to obtain the desired result.
\end{proof}

We can now generalize Propositions \ref{proposition4.6}
and~\ref{proposition4.8} to weight~$3$.

\begin{proposition}
\label{proposition4.11}
Make the same hypotheses as in Lemma~\ref{lemma4.10}, and let $s \in
\Z$.  Then we have the following congruence modulo $\Espace_3$:
\begin{equation}
\label{equation4.42}
\begin{split}
&\sum_{T \in \E[n]} x_{A \oplus T} \thelambda_{B \ominus [s]T}\\
&\quad \equiv
\bigl( \text{a linear combination of terms of the form } 
         x_{[a]A \oplus [b]B} \thelambda_{[c]A \oplus [d]B}
\bigr),\\
&\qquad\qquad\text{with }
  a - sb \equiv c - sd \equiv 0 \pmod{n}.\\
\end{split}
\end{equation}
(An analogous statement holds for sums
$\sum_T x_{A \ominus [s]T} \thelambda_{B \oplus T}$, in which case the
congruence condition modulo $n$ becomes $-sa + b \equiv -sc + d \equiv
0$.)
In contrast to~\eqref{equation4.17}, some terms above may have
$\det \stwomatr{a}{b}{c}{d} \neq \pm n$.

Furthermore, if $\pointP, \pointQ \in \E[n\myl]$, then the trace of
$\xP \lambdaQ \in \ringR_{n\myl}$ down to
$\Gammal$ is congruent modulo $\Espace_3$ to a linear combination of
terms $\xR \lambdaS \in \ringRl$, with $\pointR, \pointS \in \El$.
\end{proposition}
\begin{proof}
The proof of~\eqref{equation4.42} follows the same lines as the proof
of Proposition~\ref{proposition4.6}, by a similar induction
on $s$.  For $s=0$, we use~\eqref{equation4.10}, and
we have already proved the case $s=1$ in~\eqref{equation4.40}.  The
key step in the induction (analogous to~\eqref{equation4.19}) amounts
to applying~\eqref{equation4.40} to the $T$-part of the sum
$\sum_{T,U} x_{A\oplus T} \thelambda_{B' \ominus T \oplus U}$.  The
ideas are essentially the same as before, with the use
of~\eqref{equation4.39} thrown in for good measure.  It is worth
pointing out that while carrying out the same proof in the case of
$\sum_T x_{A \ominus [s]T} \thelambda_{B \oplus T}$, we encounter
the sum
$\sum_{U \in \E[s]}
    x_{[n]A' \oplus [n]U} \thelambda_{[n]A' \oplus [n]U}$,
where $[s]A' = A$.  Write $d=gcd(n,s)$ and $\hat{s} = s/d$;
then the sum over $U$ can be rewritten as
$d^2 \cdot \sum_{\hat{U} \in \E[\hat{s}]}
   x_{[n]A' \oplus \hat{U}} \thelambda_{[n]A' \oplus \hat{U}}$,
which we simplify using~\eqref{equation4.39}.

As for the proof of the statement about the trace of $\xP\lambdaQ$, it
follows the argument of Proposition~\ref{proposition4.8} with only
trivial changes.  The only point worth mentioning is that the roles of
$T_0$ and $U_0$ are no longer symmetric, so we cannot simply assume
that $T_0$ in case (i) (respectively, $\hat{T_0}$ in case (ii)) is not
equal to $\Pzero$.  However, if $T_0$ (respectively, $\hat{T_0}$) is equal
to $\Pzero$, then $\pointP \in \El$ already, and the trace is then equal
to $\xP \tr(\lambdaQ)$, which is easy to analyze
using~\eqref{equation4.10}.
\end{proof}

\section{Generating all modular forms in weights $\geq 2$, and a model 
  for $X(\myl)$}
\label{section5}

The main result of this section, Theorem~\ref{theorem5.1} below, is
that the ring $\ringRl$ contains all modular forms
on $\Gammal$ in weights $2$ and above.  Propositions \ref{proposition4.8}
and~\ref{proposition4.11} play a key role in the proof.
The result yields a general method to find explicit
models for the modular curve $X(\myl)$, in Theorem~\ref{theorem5.5}
below.

We prove Theorem~\ref{theorem5.1} via the
nonvanishing of a special value of an $L$-function, which is also the
strategy
of~\cite{BorisovGunnellsNonvanishing,BorisovGunnellsHigherWeight}.
Our proof brings in the $L$-function via a Rankin-Selberg integral, in
contrast to the approach of Borisov and Gunnells, which involves
$q$-expansions whose coefficients are modular symbols.  It is worth
noting that one can give a much simpler proof of the (rather weaker)
fact that $\ringRl$ contains all modular forms in sufficiently high
weights.  To see this, note first that the ring of all modular forms
is the graded integral closure of $\ringRl$ in their common field of
fractions $\fieldbigKl$. 
(This is a pleasant exercise; part of the proof involves observing that
$\fieldbigKl$ contains $\acoeff$, $\bcoeff$, and all the $\xP$s and 
$\yP$s, which, by Proposition~6.1 of \cite{Shimura}, suffice to
generate the function field of $X(\myl)$ via weight~$0$ meromorphic
ratios of elements of $\ringRl$.)  Hence
$X(\myl) = \mathbf{Proj}~\ringRl$; since $X(\myl)$ is nonsingular, it is 
then a standard fact that the graded components of the two rings ($\ringRl$
and the ring of modular forms) agree in sufficiently high weights --- see
for example~\cite{Hartshorne}, Section II.5.19 and Exercises II.5.9,
II.5.14.  Precise but large bounds for the meaning of ``sufficiently high''
for arbitrary  curves are given in~\cite{GrusonLazarsfeldPeskine}, but
they of course grow with the genus of the curve, which for $X(\myl)$
is $O(\myl^3)$.  The interest of our results, as well as those of
Borisov-Gunnells, is that they give a fixed value for ``sufficiently
high'': $2$ in our result for $\Gammal$, and $3$ for their result for
$\Gamma_1(\myl)$ (where they obtain all cusp forms modulo Eisenstein
series, but potentially miss some Eisenstein series).

\begin{theorem}
\label{theorem5.1}
Let $\fieldk = \C$.  Then $\ringRl$ contains all modular forms on $\Gammal$
of weight $2$ and above.  In other words, $\ringRl$ ``misses'' precisely
the cusp forms in weight~$1$.
\end{theorem}
\begin{proof}
Since $\ringRl$ contains all modular forms for $\myl \leq 2$, we as usual
restrict to the case $\myl \geq 3$.
Our first claim is that it is enough to show that $\ringRl$ contains all
of $\modforms_2(\Gammal)$ and $\modforms_3(\Gammal)$.  To see this,
observe that $\Gammal$ has no elliptic elements or irregular cusps;
hence there exists a line bundle $\LL$ on $X(\myl)$ such that for all
$\wtj$, we have 
$\modforms_\wtj(\Gammal) = H^0(X(\myl),\LL^{\tensor \wtj})$.  Moreover,
elements of $\modforms_2$ can be viewed as $1$-forms on $X(\myl)$ with at
worst a simple pole at each cusp.  Hence the degree of $\LL^{\tensor 2}$ is
equal to $2g-2 + \kappa$, where $g$ is the genus of
$X(\myl)$, and $\kappa$ is the number of cusps.  Since $\kappa \geq 4$ for
$\myl \geq 3$, by standard formulas for modular curves (e.g., Section~1.6
of~\cite{Shimura}), we obtain that $2 \deg \LL \geq 2g + 2$.  This is
enough to imply that the multiplication map $\modforms_{\wtj}(\Gammal)
\tensor \modforms_{\wtj'}(\Gammal) \to \modforms_{\wtj+\wtj'}(\Gammal)$ is
surjective for $\wtj, \wtj' \geq 2$, since the degrees of
$\LL^{\tensor \wtj}$ and $\LL^{\tensor \wtj'}$ are both $\geq 2g+1$ (for a
sketch of this standard result, see Lemma~2.2 of~\cite{KKM}; the survey
in Section~1 of~\cite{Lazarsfeld} is also a particularly useful
reference).  Hence any ring of modular forms containing
$\modforms_2(\Gammal)$ and $\modforms_3(\Gammal)$ must contain all forms in
higher weights.
Since $\ringRl$ contains all Eisenstein series on
$\Gammal$, we are reduced to checking whether $\ringRl$ contains all of
$\cuspforms_\wtj(\Gammal)$ for $\wtj \in \{2,3\}$, or alternatively to
checking that the orthogonal complement $[\ringRl \intersect
\cuspforms_\wtj(\Gammal)]^\perp$ in $\cuspforms_\wtj(\Gammal)$ with
respect to the Petersson inner product is zero.

We study this orthogonal complement
using a result of Shimura~\cite{ShimuraSpecialValues}, namely that a
suitable Rankin-Selberg convolution of a newform $F$ with an
Eisenstein series gives a product of two special values of Hecke  
$L$-functions of $F$ twisted by Dirichlet characters $\xi, \psi$.
More precisely, Theorem~2 (with $r=0$) of~\cite{ShimuraSpecialValues}, and
equation~(4.3) of that article (with $k=\wtj \geq 2$, $l=1$, and
$m=\wtj-1$) imply the following statement for any
$\wtj \geq 2$: 
let $F \in \cuspforms_\wtj$ be a newform with character $\chi$, and let 
$\xi, \psi$ be Dirichlet characters with $(\xi\psi)(-1) = -1$; then there
exists a product $\Eis \Eis'$ of two Eisenstein series, with $\Eis \in
\Espace_1$ and $\Eis' \in \Espace_{\wtj-1}$, such that
\begin{equation}
\label{equation5.1}
\langle F, \Eis \Eis'\rangle = C \cdot L(\wtj-1, F, \xi) L(\wtj-1, F, \psi)  
\end{equation}
with an explicit nonzero constant $C$.  (Here, if $\wtj=3$, we must have
$\chi \xi \psi \neq 1$ in order for $\Eis' \in \Espace_2$ to be
holomorphic.)  Note that we have normalized the Petersson inner product so
that it is insensitive to the choice of common congruence subgroup $\Gamma$
with respect to which $F$, $\Eis$, and $\Eis'$ are all invariant.

We deduce from~\eqref{equation5.1} that for a given $F$, we can choose
$\xi$ and $\psi$ (and, with them, $\Eis$ and $\Eis'$) so as to make
$\langle F, \Eis \Eis' \rangle \neq 0$.
Indeed, when $\wtj \geq 3$, then the $L$-functions on the right side
are nonzero for arbitrary $\xi, \psi$, since
they are evaluated outside the critical strip if $\wtj \geq 4$, and at
the edge of the critical strip if $\wtj = 3$ (see, e.g., Proposition~2
of~\cite{ShimuraSpecialValues}, or more
generally~\cite{JacquetShalikaNonvanishing}).  Thus we can also ensure
that $\chi\xi\psi \neq 1$ as
needed when $\wtj=3$.  On the other hand, if $\wtj=2$, then,
by Theorem~2 of~\cite{ShimuraPeriods}, there exist $\xi$ and $\psi$ for
which the right side of~\eqref{equation5.1} is nonzero.

Now assume there exists a nonzero cuspform $0 \neq f \in
\cuspforms_\wtj(\Gammal)$ in the orthogonal complement
$[\ringRl \intersect \cuspforms_\wtj(\Gammal)]^\perp$.
Then there exist constants $c_1,\dots, c_N \in \C$
and matrices $\alpha_1, \dots, \alpha_N \in GL^+(2,\Q)$ such that the
linear combination $F = \sum_i c_i f| \alpha_i$ is a newform (for
instance, use an element of the Hecke algebra to project to a single
automorphic representation, and then move around within it to reach the
newform).  We can find $\Eis, \Eis'$ as above such that
$\langle F, \Eis\Eis' \rangle \neq 0$.  But then
\begin{equation}
\label{equation5.2}
 0 \neq \langle \sum_i c_i f| \alpha_i, \Eis \Eis' \rangle  
 = \sum_i c_i \langle f, (\Eis \Eis')|\alpha_i^{-1} \rangle.
\end{equation}
In the above expression, each form $(\Eis \Eis')|\alpha_i^{-1}
= (\Eis|\alpha_i^{-1})(\Eis'|\alpha_i^{-1})$ is still
the product of an Eisenstein series of weight $1$ with an Eisenstein series
of weight $j-1 \in \{1,2\}$; hence it can be written as a linear
combination of modular forms of the form $\lambdaP \lambdaQ$ or $\lambdaP
\xQ$, with $\pointP, \pointQ \in \E[n\myl]$ for some (possibly rather
large) $n$.  We obtain a linear combination of inner products of the
form $\langle f, \lambdaP \lambdaQ \rangle$ or
 $\langle f, \lambdaP \xQ \rangle$, which can in turn be
reexpressed (up to a constant factor) as an inner product of the form
$\langle f, \tr^{\Gamma(n\myl)}_\Gammal \lambdaP \lambdaQ\rangle$
or $\langle f, \tr^{\Gamma(n\myl)}_\Gammal \lambdaP \xQ\rangle$,
and the traces belong to $\ringRl$ by Propositions \ref{proposition4.8}
and~\ref{proposition4.11}.  Furthermore, the Eisenstein part of each
such trace, and therefore also the cuspidal part, must then belong to
$\ringRl$.  Thus the inner products must all be zero if $f$ belongs to the
orthogonal complement $[\ringRl \intersect \cuspforms_\wtj(\Gammal)]^\perp$
in question.  This contradicts the fact that $\langle F, \Eis \Eis' \rangle
\neq 0$, and we deduce that the orthogonal complement is zero after all.
This concludes our proof.
\end{proof}

Theorem~\ref{theorem5.1} allows us to compute nice
models for the modular curve $X(\myl)$.  These models
are in the form called ``Representation~B'' in~\cite{KKMasymptotic},
which we now describe, along with the related ``Representation~A''.
In our application, $X = X(\myl)$, while $\LLhat = \LL^{\tensor2}$, in
the notation of Theorem~\ref{theorem5.1}; 
thus $V = \modforms_2(\Gammal)$ and $V' = \modforms_4(\Gammal)$. 

\begin{definition}
\label{definition5.2}
Let $X$ be a smooth genus $g$ projective curve over a perfect base
field $F$, and let $\LLhat$ be an $F$-rational line bundle on $X$ with
$\deg \LLhat \geq 2g+2$.  Define vector spaces $V, V'$ and 
multiplication maps $\mu, \mubar$ (where $\mu$ factors through
$\mubar$) by 
\begin{equation}
\label{equation5.3}
V = H^0(X,\LLhat),
\quad
V' = H^0(X,\LLhat^{\tensor2}),
\quad
\mu: V \tensor V \to V',
\quad
\mubar: \Sym^2 V \to V'.
\end{equation}
It is a standard fact that $\LLhat$ gives rise to a projective
embedding of $X$ into $\Projective^L$ ($L = \deg \LLhat - g$),
such that the ideal of equations defining the image of 
$X$ is generated by quadrics (see, for example, Section~1
of~\cite{Lazarsfeld}).  These quadrics correspond precisely to the
kernel of $\mubar$.  We thus define Representation~A of the curve $X$
to be a pair of identifications $V \isomorphic F^{L+1}$ and $V'
\isomorphic F^{L'}$, with the map $\mu$ described by a
multiplication table in terms of the coordinates on $F^{L+1}$ and
$F^{L'}$.  This suffices to determine $\ker \mubar$, and with it
defining equations for $X$ in $\Projective^L$.

As for Representation~B, we do away with the description of $\mu$,
essentially by interpolating through sufficiently many points of
$X(\overline{F})$.
Specifically, let $\divD$ be an $F$-rational effective divisor on $X$,
with $N = \deg \divD > 2\deg \LLhat$.  (The reader should imagine
that $\divD = \sum_{i=1}^N (\p_i)$ with the $\{\p_i\}$ distinct points in
$X(\overline{F})$, but in general points may occur with multiplicity;
see Remark~\ref{remark5.3} below.)  By our choice of $\deg \divD$,
evaluation of sections ``at $\divD$'' gives injections of $F$-vector spaces
\begin{equation}
\label{equation5.4}
H^0(X,\LLhat) \hookrightarrow H^0(X,\LLhat/\LLhat(-\divD)),
\quad
H^0(X,\LLhat^{\tensor2}) \hookrightarrow
H^0(X,\LLhat^{\tensor2}/\LLhat^{\tensor2}(-\divD)).
\end{equation}
Moreover, one can (by choosing a suitable trivialization near $\divD$)
identify each of $H^0(X,\LLhat/\LLhat(-\divD))$ and 
$H^0(X,\LLhat^{\tensor2}/\LLhat^{\tensor2}(-\divD))$ with the
($\deg \divD$)-dimensional $F$-algebra $\mathcal{A} =
H^0(X,\mathcal{O}_\divD)$, in such a way that $\mu$ is induced by the
multiplication on $\mathcal{A}$.  We then define Representation~B
of $X$ to be the data of the algebra $\mathcal{A}$, along with
subspaces $V,V' \subset \mathcal{A}$.

We describe concretely what this means if $\divD = \sum_i (\p_i)$ with
distinct $\p_i \in X(F)$ (the case where only the divisor $\divD$
is rational over $F$, but the individual points are not, is a
``descent'' of this).  We then have $\mathcal{A} \isomorphic F^N =
F\times \dots \times F$, the split $N$-dimensional algebra.  Given a
trivialization of $\LLhat$ near the $\{p_i\}$, the injections
$V,V' \hookrightarrow \mathcal{A}$ identify a section $s$ 
with its vector of ``values'' $(s(\p_1), \dots, s(\p_n)) \in
\mathcal{A}$.  Thus each of $V,V'$ is a subspace of $F^N$, and
$\mu$ is componentwise multiplication.

We now relate this to the projective embedding.  
Let $\{s_0, \dots, s_L\}$ be a basis for $V$;
then each point $\p_i \in X$ maps to the projective point
$\p'_i = [s_0(\p_i):\cdots:s_L(\p_i)] \in \Projective^L$.  (Thus if
one considers the matrix $(s_j(\p_i))_{ij}$, the rows give points on
the image of $X$, while the columns span the subspace $V \subset
\mathcal{A}$.)   
Then the image of $X$ is the unique projective curve that interpolates
the $\{\p'_i\}$, in the sense that its ideal is generated by the quadric
equations vanishing at the $\{\p'_i\}$.  These quadrics are of the
form $\sum_{j,k} c_{jk} X_j X_k$, and can be found by solving for the
$c_{ij}$ in the linear system
$\{\sum_{j,k} c_{jk} s_j(\p_i) s_k(\p_i) = 0 \mid 1\leq i \leq N\}$.
Even if the individual $\p_i$ are not defined over
$F$, the set of points $\{\p'_i\}$ is stable under
$\Gal(\overline{F}/F)$, and so the linear system of equations for the
$c_{jk}$ is unaffected by the Galois group.  So even if we do the
linear algebra over $\overline{F}$, the echelon basis for the
solution space of the linear system will be defined over $F$, and we
will obtain $F$-rational quadrics that define the image of $X$.
\end{definition}

\begin{remark}
\label{remark5.3}
In our desired application, $V$ and $V'$ are spaces of modular forms,
and one option is to take
$\divD = N\cdot\infty$, where the cusp $\infty$ is assumed $F$-rational;
then $\mathcal{A} = F[[q]]/(q^N)$, and the vector of values of each
section $s_j \in \modforms_2(\Gammal)$ is a truncated $q$-expansion.
Linear relations between $q$-expansions of products $s_j s_k$ then
give rise to equations for $X(\myl)$, as discussed above. 
This approach has already appeared in the literature;
see~\cite{Galbraiththesis} and Section~2 of~\cite{BakerGonzalez2Poonen},
which instead use $\cuspforms_2(\Gamma)$ and products in
$\cuspforms_4(\Gamma)$ to obtain the canonical embedding of
$X(\Gamma)$ in most cases.  Occasionally $X$ is hyperelliptic, or the
canonical curve is not defined by quadrics, in which case they go to
higher weights.

One novel aspect of our approach is that we evaluate the modular forms at
noncuspidal points; we hope that this approach, suitably developed, can
eventually also yield equations of Shimura curves.
\end{remark}

We are now ready for the last result of this article.

\begin{theorem}
\label{theorem5.5}
Let $\myl \geq 3$.  Fix a number field $F \subset \C$ and an elliptic
curve $\E^\zirc$ over $F$ given by a Weierstrass equation
$y^2 = x^3 + a^\zirc x + b^\zirc$, with 
$a^\zirc, b^\zirc \in F - \{0\}$.  Then consider all torsion points 
$\{(x^\zirc_{\pointP}, y^\zirc_{\pointP})
     \mid \pointP \in \E^\zirc[\myl](\overline{F}) - \{\Pzero\}\}$,
and the slopes 
$\thelambda^\zirc_{(\pointP)+(\pointQ)+(\ominus \pointP \ominus \pointQ)} =
(y^\zirc_{\pointP} - y^\zirc_{\pointQ})/(x^\zirc_{\pointP} - x^\zirc_{\pointQ})
      \in F(\E^\zirc[\myl])$
of lines through pairs of torsion points (with the appropriate modification
when $\pointP = \pointQ$).  These slopes for the one elliptic curve $\E^\zirc$
contain enough information to reconstruct the projective embedding of
$X(\myl)$ coming from the linear system $\modforms_2(\Gammal)$.  This
embedding is defined over $F(\boldmu_\myl)$.
\end{theorem}
\begin{proof}
The condition $a^\zirc, b^\zirc \neq 0$ implies that $\E^\zirc$ 
does not correspond to an elliptic point for $\Gamma(1)$ in the upper
half plane $\Half$.  Thus the projection map $\pi: X(\myl) \to  X(1)$
is unramified over the point  $q^\zirc \in X(1)(F)$ corresponding to 
$\E^\zirc$, and hence the preimages
$\{\p_1, \dots, \p_N\} = \pi^{-1}(\{q^\zirc\})$
are distinct points of $X(\myl)$, which are rational over the
field $F_\myl = F(\E^\zirc[\myl])$.  We claim that $N$
(which is $\abs{PSL(2,\Z/\myl\Z)}$)
is large enough that we can identify modular forms of weight~$< 12$
via their ``values'' at the $\{\p_i\}$.  To see this claim, either use
standard formulas for the degree of the line bundle $\LL^{\tensor
  \wtj}$, whose global sections are $\modforms_\wtj(\Gammal)$, or note
that one section of the line bundle $\LL^{\tensor 12}$ is the
$\Gamma(1)$-invariant modular form
$(b^\zirc)^2 a(\tau)^3 - (a^\zirc)^3 b(\tau)^2$, 
which vanishes precisely to order~$1$ at each point $\p_i$; 
indeed, modular forms in $\modforms_{12}(\Gamma(1))$
have precisely one zero (counted appropriately) in the fundamental domain
for the $\Gamma(1)$-action on~$\Half$.  Thus $N = 12 \deg \LL$, and
our claim is proved.

Hence, as described earlier, we can
represent $X(\myl)$ in Representation~B using the line bundle 
$\LLhat = \LL^{\tensor 2}$ and the divisor $\divD = \sum_i (\p_i)$; 
thus we represent the spaces $V, V'\subset F_\myl^N$ by vectors of
``values'' of modular 
forms of weights $2$ and~$4$ at the points $\{\p_i\}$.  Concretely, such a
point $\p_i$ corresponds to a choice of symplectic basis
$\{T,U\}$ for the $\myl$-torsion $\E^\zirc[\myl]$, with
$e_\myl(T, U) = e^{2\pi i/\myl} \in F_\myl$.  We know
how to ``evaluate'' an Eisenstein series of weight~$1$ at $\p_i$:
just compute slopes between the torsion points to get each
$\lambda^\zirc \in F_\myl$ in the statement of the theorem.  Here, the
local trivialization of each line bundle $\LL^{\tensor \wtj}$ near
$\p_i$ corresponds to the particular choice of Weierstrass model of
$\E^\zirc$ and of its global differential $\omega^\zirc$.  To define
this trivialization more precisely, let $\tau_{T,U} \in \Half$ be such
that the elliptic curve 
$\E_{\tau_{T,U}} = \C/L_{\tau_{T,U}}$ and its symplectic $\myl$-torsion
basis $\{\pointP_{1/\myl}, \pointP_{\tau_{T,U}/\myl}\}$ are isomorphic
to our given triple $(\E^\zirc, T, U)$.
(The $\tau_{T,U}$
all belong to a single $\Gamma(1)$-orbit, determined by $\E^\zirc$.)
Then there exists a unique $\scalarlambda \in \C^\times$ such 
that $a^\zirc = \scalarlambda^4 \acoeff(\tau_{T,U})$ and 
$b^\zirc = \scalarlambda^6 \bcoeff(\tau_{T,U})$,
with a similar compatibility between the level structures.
Hence each $\lambda^\zirc$ is equal to $\scalarlambda
\lambda_1(\tau_{T,U})$ for a corresponding classical modular form 
$\lambda_1(\tau) \in \Espace_1^{\myl}$, and similarly for other
weights~$\wtj$.  It follows that our trivialization of $\LL^{\tensor \wtj}$
near $\p_i$ is $u^\wtj$ times the trivialization induced by evaluating
modular forms in a neighborhood of $\tau_{T,U}$.

Hence (at least over $F_\myl$), we have $\mathcal{A} \isomorphic
F_\myl^N$,  and we compute the subspace $V$ (respectively, $V'$)
as the span of all componentwise products of vectors of values of two
(respectively, four) of the $\lambda^\zirc$s at each $p_i$.  This
follows from Theorems \ref{theorem5.1} and~\ref{theorem3.12}.  We thus
obtain equations for $X(\myl)$ from interpolation, or equivalently from 
$\ker (\mubar:\Sym^2 V \to V')$.  These equations are actually defined
over the smaller cyclotomic extension $F(\boldmu_\myl)$, because
$\Gal(F_\myl/F(\boldmu_\myl))$ acts via a subgroup of $SL(2,\Z/\myl\Z)$ on
symplectic bases for $\E^\zirc[\myl]$, thereby permuting the
points $\{\p_i\}$ in $\divD$.  As a final computational note, one
should not in practice 
list the value in $F_\myl$ at every single $\p_i$ or carry out the
linear algebra over the $F_\myl$: instead, one remains over
$F(\boldmu_\myl)$, in 
which case Galois conjugates of the $\p_i$ get lumped together, and
$\mathcal{A}$ becomes an \'etale $F(\boldmu_\myl)$-algebra.  One
evaluates each 
$\lambda^\zirc$ at a single ``virtual'' symplectic basis $\{T,U\}$ in
$\E^\zirc[\myl](\mathcal{A})$, which yields a value in $\mathcal{A}$;
the products (in $\mathcal{A}$) of pairs of these values span $V$ over
$F(\boldmu_\myl)$.  The quadric equations defining $X(\myl)$
are then the $F(\boldmu_\myl)$-linear relations between the various
products of pairs of elements of a basis for $V$.
\end{proof}

We note in closing that an analog of Theorem~\ref{theorem5.5} holds
for the projective embedding of $X(\myl)$ coming from the (usually
incomplete) linear system 
$\Espace_1^\myl \subset \modforms_1(\Gammal)$.  By
Theorem~\ref{theorem5.1} and a computation of Castelnuovo-Mumford
regularity, that projective model is defined by equations in degrees
$2$ and~$3$.


\providecommand{\bysame}{\leavevmode\hbox to3em{\hrulefill}\thinspace}
\providecommand{\MR}{\relax\ifhmode\unskip\space\fi MR }
\providecommand{\MRhref}[2]{%
  \href{http://www.ams.org/mathscinet-getitem?mr=#1}{#2}
}
\providecommand{\href}[2]{#2}


\end{document}